\documentclass[11pt]{amsart}
\usepackage{fullpage,url,amssymb}
\usepackage{mathrsfs}
\numberwithin{equation}{section}
\usepackage{color}
\usepackage{booktabs,array}
\usepackage{tikz-cd}

\newcommand{\defi}[1]{\emph{#1}} 

\newenvironment{romanenum}{\hfill \begin{enumerate} }{\end{enumerate}}
\newenvironment{alphenum}{\hfill \begin{enumerate} }{\end{enumerate}}

\newcommand{\PP}{{\mathbb P}}
\newcommand{\FF}{{\mathbb F}}

\newcommand{\ZZ}{{\mathbb Z}}
\newcommand{\Zhat}{{\hat{\ZZ}}}

\def\bbar#1{\setbox0=\hbox{$#1$}\dimen0=.2\ht0 \kern\dimen0 \overline{\kern-\dimen0 #1}}
\newcommand{\Qbar}{{\overline{\mathbb Q}}}

\newcommand{\calA}{{\mathcal A}}

\newcommand{\calF}{{\mathcal F}}
\newcommand{\calG}{{\mathcal G}}

\newcommand{\calJ}{{\mathcal J}}

\newcommand{\calS}{{\mathcal S}}

\newcommand{\OO}{{\mathcal O}}

\DeclareMathOperator{\Aut}{Aut}
\DeclareMathOperator{\Gal}{Gal}

\DeclareMathOperator{\ord}{ord}

\DeclareMathOperator{\Spec}{Spec}

\newcommand{\GL}{\operatorname{GL}}
\newcommand{\SL}{\operatorname{SL}}

\newcommand{\Mat}{\operatorname{M}}

\def\QQ{\mathbb Q}
\def\RR{\mathbb R}
\def\CC{\mathbb C}
\def\HH{\mathbb H}

\newcommand{\smallmat}[4]{\left(\begin{smallmatrix}#1&#2\\#3&#4\end{smallmatrix}\right)}

\newtheorem{theorem}{Theorem}[section]
\newtheorem{lemma}[theorem]{Lemma}
\newtheorem{corollary}[theorem]{Corollary}
\newtheorem{proposition}[theorem]{Proposition}
\theoremstyle{definition}

\newtheorem{example}[theorem]{Example}

\theoremstyle{remark}
\newtheorem{remark}[theorem]{Remark}

\definecolor{webcolor}{rgb}{0,0,1}
\definecolor{webbrown}{rgb}{.6,0,0}
\usepackage[
       	colorlinks,
       	linkcolor=webbrown,  filecolor=webcolor,  citecolor=webbrown, 
       	backref,
       	backref,
       	pdfauthor={},
	pdftitle={},
]{hyperref}
\usepackage[alphabetic,backrefs,lite]{amsrefs} 

\begin{document}
\title[Modular curves of prime-power level with infinitely many rational points]{Modular curves of prime-power level\\with infinitely many rational points}
\subjclass[2000]{Primary 14G35; Secondary 11G05, 11F80}

\author{Andrew V. Sutherland}
\address{Department of Mathematics, Massachusetts Institute of Technology, Cambridge, MA 02139, USA}
\email{drew@math.mit.edu}
\urladdr{http://math.mit.edu/~drew}

\author{David Zywina}
\address{Department of Mathematics, Cornell University, Ithaca, NY 14853, USA}
\email{zywina@math.cornell.edu}
\urladdr{http://www.math.cornell.edu/~zywina}

\thanks{The first author was supported by NSF grants DMS-1115455 and DMS-1522526.}

\begin{abstract}
For each open subgroup $G$ of $\GL_2(\Zhat)$ containing $-I$ with full determinant, let $X_G/\QQ$ denote the modular curve that loosely parametrizes elliptic curves whose Galois representation, which arises from the Galois action on its torsion points, has image contained in $G$.   Up to conjugacy, we determine a complete list of the $248$ such groups $G$ of prime power level for which $X_G(\QQ)$ is infinite. For each $G$, we also construct explicit maps from each $X_G$ to the $j$-line.
This list consists of $220$ modular curves of genus $0$ and $28$ modular curves of genus~$1$.  For each prime $\ell$, these results provide an explicit classification of the possible images of $\ell$-adic Galois representations arising from elliptic curves over $\QQ$ that is complete except for a finite set of exceptional $j$-invariants.
\end{abstract}

\maketitle

\section{Introduction}\label{S:intro}

Let $E$ be an elliptic curve defined over $\QQ$ and denote its $j$-invariant by $j_E$.  For each positive integer $N$, let $E[N]$ denote the $N$-torsion subgroup of $E(\Qbar)$, where $\Qbar$ is a fixed algebraic closure of $\QQ$.  The natural action of the absolute Galois group $\Gal_\QQ:=\Gal(\Qbar/\QQ)$ on $E[N]\simeq (\ZZ/N\ZZ)^2$ induces a Galois representation
\[
\rho_{E,N}\colon \Gal_\QQ\to  \GL_2(\ZZ/N\ZZ).
\]
After choosing compatible bases for the torsion subgroups $E[N]$, these representations determine a Galois representation
\[
\rho_E\colon \Gal_\QQ\to \GL_2(\Zhat)
\]
whose composition with the projection $\GL_2(\Zhat)\to\GL_2(\ZZ/N\ZZ)$ given by reduction modulo $N$ is equal to $\rho_{E,N}$, for each $N$.
The images of $\rho_{E,N}$ and $\rho_E$ are uniquely determined up to conjugacy in $\GL_2(\ZZ/N\ZZ)$ and $\GL_2(\Zhat)$, respectively.  If $E$ does not have complex multiplication (CM), then $\rho_E(\Gal_\QQ)$ is an open subgroup of $\GL_2(\Zhat)$, by Serre's open image theorem \cite{MR0387283}, hence of finite index in $\GL_2(\Zhat)$.

Let $G$ be an open subgroup of $\GL_2(\Zhat)$ that satisfies $\det(G)=\Zhat^\times$ and $-I\in G$.
Let $N$ be the least positive integer such that $G$ is the inverse image of its image under the reduction map $\GL_2(\Zhat)\to \GL_2(\ZZ/N\ZZ)$; we call $N$ the \defi{level} of $G$.

Associated to $G$ is a modular curve $X_G/\QQ$; one can define $X_G$ as the generic fiber of the smooth proper $\ZZ[1/N]$-scheme that is the coarse moduli space for the algebraic stack $\mathcal{M}_{\bbar G}[1/N]$ in the sense of \cite{MR0337993}*{\S IV}, where $\bbar G$ denotes the image of $G$ under reduction modulo $N$.  See \S\ref{S:modular} for some background on $X_G$ and an alternate description; in particular, it is a smooth projective geometrically integral curve defined over $\QQ$.

When $G=\GL_2(\Zhat)$, the modular curve $X_G$ is the $j$-line $\PP^1_\QQ=\mathbb{A}^1_\QQ \cup \{\infty\}
$.
If $G$ and $G'$ are open subgroups of $\GL_2(\Zhat)$ with $\det(G)=\det(G')=\Zhat^\times$ and $-I\in G,G'$ such that $G\subseteq G'$, then there is a natural morphism $X_G\to X_{G'}$ of degree $[G':G]$.  In particular, with $G'=\GL_2(\Zhat)$, we have a morphism 
\[
\pi_G\colon X_G\to \PP^1_\QQ=\mathbb{A}^1_\QQ \cup \{\infty\}
\]
of degree $[\GL_2(\Zhat):G]$ from $X_G$ to the $j$-line.

The key property for our applications is that for an elliptic curve $E/\QQ$ with $j_E\notin\{0,1728\}$, the group $\rho_E(\Gal_\QQ)$ is conjugate in $\GL_2(\Zhat)$ to a subgroup of $G$ if and only if $j_E$ is an element of $\pi_G(X_G(\QQ))$; see Proposition~\ref{P:XG rational 2}.   This property requires $-I\in G$, since there is always an elliptic curve $E$ with any given rational $j$-invariant such that $-I \in \rho_E(\Gal_\QQ)$; it also requires $\det(G)=\Zhat^\times$, since $\det(\rho_E(\Gal_\QQ))=\Zhat^\times$, and that $G$ contain an element corresponding to complex conjugation.

We are interested in those groups $G$ for which $X_G$ has infinitely many rational points; equivalently, for which there are infinitely many elliptic curves $E/\QQ$, with distinct $j$-invariants, such that $\rho_E(\Gal_\QQ)$ is conjugate to a subgroup of $G$.  We need only consider modular curves $X_G$ of genus $0$ or $1$ since otherwise $X_G(\QQ)$ is finite by Faltings' Theorem \cite{MR718935}.

In this article, we give an explicit description of all such subgroups $G\subseteq \GL_2(\Zhat)$ for which the modular curve $X_G$ has infinitely many rational points in the special case where the level $N$ of $G$ is a \emph{prime power};  we also give an explicit model for $X_G$ and the morphism $\pi_G$.   We need only describe the groups $G$ up to conjugacy in $\GL_2(\Zhat)$.  For notational simplicity, we define the genus of $G$ to be the genus of the corresponding curve $X_G$.

\begin{theorem} \label{T:main}
Up to conjugacy, there are $248$ open subgroups $G$ of $\GL_2(\Zhat)$ of prime power level satisfying $-I\in G$ and $\det(G)=\Zhat^\times$ for which $X_G$ has infinitely many rational points.   Of these $248$ groups, there are $220$ of genus $0$ and $28$ of genus $1$.
\end{theorem}

The $220$  subgroups of genus $0$ in Theorem~\ref{T:main} are given in Tables~\ref{table:g0three}, \ref{table:g0odd} and \ref{table:g0two} in Appendix~\ref{S:appendix}.   For such a group $G$ of genus $0$, we also describe the morphism $\pi_G$.  More precisely, we give a rational function $J(t)\in \QQ(t)$ such that the function field of $X_G$ is of the form $\QQ(t)$ and the morphism from $X_G$ to the $j$-line is given by the equation $j=J(t)$.   In particular, if $E/\QQ$ is an elliptic curve with $j_E\notin \{0,1728\}$, then $\rho_{E}(\Gal_\QQ)$ is conjugate to a subgroup of $G$ if and only if $j_E=J(t_0)$ for some $t_0\in \QQ\cup\{\infty\}$.

The $28$ subgroups of genus $1$ in Theorem~\ref{T:main} are listed in Table~\ref{table:g1} of Appendix~\ref{S:appendix}; their levels are all powers of $2$ except for a group of level $11$ whose image in $\GL_2(\ZZ/11\ZZ)$ is the normalizer of a non-split Cartan subgroup.  For such a group $G$ of genus $1$, we give a Weierstrass model for $X_G$ and the morphism $\pi_G$ to the $j$-line.

\begin{example}
Up to conjugacy, there is a unique subgroup $G\subseteq \GL_2(\Zhat)$ of genus $0$ and level $27$ given by Theorem~\ref{T:main}.  It has label $\text{27A}^0\text{-27a}$ in our classification, and we may choose it so that the image of $G$ in $\GL_2(\ZZ/27\ZZ)$ is generated by the matrices  $\smallmat{1}{1}{0}{1},\smallmat{2}{1}{9}{5}$ and $\smallmat{1}{2}{3}{2}$.   Using Table~\ref{table:g0three}, associated to $G$ is the rational function
\[
J(t)=F_3(F_2(F_1(t))) =\frac{(t^3 + 3)^3 (t^9 + 9t^6 + 27t^3 + 3)^3}{t^3(t^6 + 9t^3 + 27 )},
\]
where $F_1(t)=t^3$, $F_2(t)=t(t^2+9t+27)$ and $F_3(t)=(t+3)^3(t+27)/t$.  That $J(t)$ is the composition of three rational functions reflects the fact that the morphism $\pi_G$ factors as $X_G\to X_{G'}\to X_{G''} \to \PP^1_\QQ$ for some groups $G \subsetneq G' \subsetneq G'' \subsetneq \GL_2(\Zhat)$.
The groups $G'$ and $G''$ have labels $\text{9B}^0\text{-9a}$ and $\text{3B}^0\text{-3a}$, respectively, and can also be found in Table~\ref{table:g0three}.
\end{example}

\begin{remark}
In contrast to the case of prime power level, in general there are infinitely many open subgroups $G$ of $\GL_2(\Zhat)$ satisfying $-I\in G$ and $\det(G)=\Zhat^\times$ for which the modular curve $X_G$ has infinitely many rational points.
Let us explicitly construct just one of several infinite families of such groups~$G$.

Let $D$ be the discriminant of a quadratic number field and let $\chi_D\colon \Zhat^\times \to \{\pm 1\}$ be the continuous quadratic character arising from the corresponding Dirichlet character.   Let $\varepsilon\colon \GL_2(\Zhat) \to \{\pm 1\}$ be the character obtained by composing the reduction map $\GL_2(\Zhat)\to\GL_2(\ZZ/2\ZZ)$ with the unique non-trivial homomorphism $\GL_2(\ZZ/2\ZZ)\to \{\pm 1\}$.   Define the group
\[
G_D:=\{ A\in \GL_2(\Zhat): \varepsilon(A)=\chi_D(\det(A))\};
\]
it is an open subgroup of $\GL_2(\Zhat)$ of index 2 containing $-I$ with $\det(G_D)=\Zhat^\times$ whose level is $|D|$ or $2|D|$, depending on whether $D\equiv 0\bmod 4$ or $D\equiv 1\bmod 4$.
For $D\neq D'$, the groups $G_D$ and $G_{D'}$ are not conjugate in $\GL_2(\Zhat)$.

The modular curve $X_{G_D}$ has genus $0$ and a rational point (it has a unique, hence rational, cusp); the function field of $X_{G_D}$ is of the form $\QQ(t)$ with the map to the $j$-line given by $J(t)=Dt^2+1728$.
Each $X_{G_D}$ is a $\QQ(\sqrt{D})$-twist of the modular curve $X_G$ corresponding to the unique index 2 subgroup $G\subseteq\GL_2(\Zhat)$ whose reduction has index 2 in $\GL_2(\ZZ/2\ZZ)$; it has label $\text{2A}^0\text{-2a}$ in our classification and can be found in Table~\ref{table:g0two}, along with its map to the $j$-line, which is $J(t)=t^2+1728$.

In general, if $\Gamma\subseteq\SL_2(\ZZ)$ is a fixed congruence subgroup of level $N$ and index $m$ containing $-I$, there will be infinitely many non-conjugate open subgroups $G\subseteq \GL_2(\Zhat)$ of index $M$ containing $-I$ with $\det(G)=\Zhat^\times$ whose reductions modulo $N$ coincide with that of $\Gamma$ upon intersection with $\SL_2(\ZZ/N\ZZ)$.
The levels $M$ of these groups $G$ may be arbitrarily large multiples of $N$ (and divisible by arbitrarily large primes).
The corresponding modular curves $X_G/\QQ$ are non-isomorphic, but for each $X_G$ there is a cyclotomic field $\QQ(\zeta_M)$ over which $X_G$ becomes isomorphic to the modular curve $X_\Gamma/\QQ(\zeta_N)$ (the quotient of the extended upper half plane by the action of $\Gamma$) after base change; as in our example, the $X_{G}$ form an infinite family of twists.
\end{remark}

\subsection{\texorpdfstring{$\ell$}{\textit{l}}-adic representations}

Fix a prime $\ell$.  Define the set
\[
\calJ_\ell:=\bigcup_G \big(\pi_G(X_G(\QQ)) \cap \QQ\big)
\] 
of rational numbers, where $G$ varies over the open subgroups of $\GL_2(\Zhat)$ whose level is a power of $\ell$ and satisfies $-I\in G$ and $\det(G)=\Zhat^\times$, and for which $X_G(\QQ)$ is finite.
Note that the set $\calJ_\ell$ contains the 13 $j$-invariants of CM elliptic curves over $\QQ$: for $n\ge 1$ each CM $j$-invariant corresponds to points on at least one of the modular curves $X_{\rm s}^+(\ell^n)$, $X_{\rm ns}^+(\ell^n)$, $X_0(\ell^n)$, and for sufficiently large $n$ these curves have genus at least $2$, hence finitely many rational points (by Faltings' Theorem).

For an elliptic curve $E/\QQ$, let 
\[
\rho_{E,\ell^\infty}\colon \Gal_\QQ\to \GL_2(\ZZ_\ell)
\]
be the representation describing the Galois action on the $\ell$-power torsion points; it is the composition of $\rho_E$ with the natural projection $\GL_2(\Zhat)\to\GL_2(\ZZ_\ell)$.   After excluding a finite number of $j$-invariants, we will describe the possible images of the $\ell$-adic representation arising from elliptic curves over $\QQ$.   Denote by $\pm \rho_{E,\ell^\infty}(\Gal_\QQ)$ the group generated by $-I$ and $\rho_{E,\ell^\infty}(\Gal_\QQ)$.  

The following theorem describes the possibilities for $\pm \rho_{E,\ell^\infty}(\Gal_\QQ)$, up to conjugacy, when $j_E$ is not in the (finite!) set $\calJ_\ell$.

\begin{theorem} \label{T:ell-adic}
\begin{romanenum}
\item \label{T:ell-adic 0}
The set $\calJ_\ell$ is finite.
\item \label{T:ell-adic i}
If $E/\QQ$ is an elliptic curve with $j_E\notin \calJ_\ell$, then $\pm \rho_{E,\ell^\infty}(\Gal_\QQ)$ is conjugate in $\GL_2(\ZZ_\ell)$ to the $\ell$-adic projection of a unique group $G$ from Theorem~\ref{T:main} with $\ell$-power level.  Moreover, $G$ does not have genus $1$, level $16$, and index $24$ in $\GL_2(\Zhat)$.
\item \label{T:ell-adic ii}
Let $G$ be a group from Theorem~\ref{T:main} with $\ell$-power level that does not have genus $1$, level $16$, and index $24$ in $\GL_2(\Zhat)$.  Then there are infinitely many elliptic curves $E/\QQ$, with distinct $j$-invariants, such that $\pm\rho_{E,\ell^\infty}(\Gal_\QQ)$ is conjugate in $\GL_2(\ZZ_\ell)$ to the $\ell$-adic projection of $G$.
\end{romanenum}
\end{theorem}

\begin{remark}
\begin{romanenum}
\item
Serre has asked whether $\rho_{E,\ell}$ is surjective for all non-CM elliptic curves $E/\QQ$ and all primes $\ell>37$, cf.~\cite{MR644559}*{p.\ 399}.  For $\ell>37$, this would imply that the set $\calJ_\ell$ consists of only the $13$ $j$-invariants of CM elliptic curves over $\QQ$.
\item
It would be nice to explicitly know the finite sets $\calJ_\ell$; the proof that $\calJ_\ell$ is finite relies on \cite{PossibleIndices}, which is ineffective since it applies Faltings' Theorem several times.
\end{romanenum}
\end{remark}

Theorem~\ref{T:ell-adic} describes the subgroups of $\GL_2(\ZZ_\ell)$, up to conjugacy, that occur as $\pm \rho_{E,\ell^\infty}(\Gal_\QQ)$ for infinitely many elliptic curves $E/\QQ$ with distinct $j$-invariants.    

Theorem~\ref{T:ell-adic} also allows us to determine the subgroups of $\GL_2(\ZZ_\ell)$, up to conjugacy, that occur as $\rho_{E,\ell^\infty}(\Gal_\QQ)$ for infinitely many elliptic curves $E/\QQ$ with distinct $j$-invariants.  They are precisely the subgroups $H$ of the $\ell$-adic projection $G$ of a group from Theorem~\ref{T:ell-adic} with $\ell$-power level such that $\pm H = G$.  Indeed if $G=\pm \rho_{E,\ell^\infty}(\Gal_\QQ)$, then for any such $H$ there is a quadratic twist of $E$ such that $H$ is conjugate to $\rho_{E',\ell^\infty}(\Gal_\QQ)$, cf.~\cite{PossibleImages}*{\S5.1} and ~\cite{ComputingImages}*{\S5.6}; when $H$ is properly contained in $G$ this quadratic twist is unique up to isomorphism and can be explicitly determined.

\begin{corollary}\label{C:ell-adic}
For $\ell=2,3,5,7,11,13$ there are respectively $1201, 47, 23, 15, 2, 11$ subgroups~$H$ of $\GL_2(\ZZ_\ell)$ that arise as $\rho_{E,\ell^\infty}(\Gal_\QQ)$ for infinitely many elliptic curves $E/\QQ$ with distinct $j$-invariants; for $\ell > 13$ the only such subgroup is $H=\GL_2(\ZZ_\ell)$.
\end{corollary}

A list of the groups $H$ appearing in Corollary~\ref{C:ell-adic} can be found in electronic form at \cite{MagmaScripts}.

\subsection{Overview}

We now give a brief overview of the contents of this paper.  As already noted, the groups $G$ from Theorem~\ref{T:main}, along with the corresponding modular curves $X_G$ and morphisms $\pi_G$, can be found in Appendix~\ref{S:appendix}

In \S\ref{S:modular}, we review the background material we need concerning the modular curves $X_G$.  If $G$ has level $N$, then we can identify the function field of $X_G$ with a subfield of the field $\calF_N$ of modular functions on $\Gamma(N)$ whose Fourier coefficients lie in the cyclotomic field $\QQ(\zeta_N)$.   As a working definition of $X_G$, we define it in terms of its function field.

In \S\ref{S:group theory}, we determine up to conjugacy the open subgroups $G$ of $\GL_2(\Zhat)$ with genus at most $1$ that satisfy $\det(G)=\Zhat^\times$, $-I\in G$ and contain an element that ``looks like complex conjugation''; this last condition is necessary, since otherwise $X_G(\RR)$, and therefore $X_G(\QQ)$, is empty.   We are left with $220$ groups of genus $0$ and $250$ groups of genus $1$ that include all the groups that appear in Theorem~\ref{T:main}.   These computations make use of the tables of Cummins and Pauli of congruence subgroups of low genus \cite{MR2016709}.   

Let $\Gamma$ be a congruence subgroup of $\SL_2(\ZZ)$ and let $X_\Gamma$ be the smooth compact Riemann surface obtained by taking the quotient of the complex upper-half plane by $\Gamma$ and adjoining cusps.  Assume further that $X_\Gamma$ has genus $0$.    In \S\ref{S:hauptmoduls}, we describe how to explicitly construct a \defi{hauptmodul} for $\Gamma$; it is a meromorphic function $h$ on $X_\Gamma$ that has a unique pole at the cusp at $\infty$.  We describe~$h$ in terms of \emph{Siegel functions}; its Fourier coefficients are computable and lie in the field $\QQ(\zeta_N)\subseteq \CC$.

In \S\ref{S:genus 0 check}, we prove the part of Theorem~\ref{T:main} concerning genus $0$ groups.  Let $G$ be one of the genus~$0$ groups from \S\ref{S:group theory} and let $J(t)\in \QQ(t)$ be the corresponding rational function from Appendix~\ref{S:appendix}.    We need to verify that the function field $\QQ(X_G)$ of $X_G$ is of the form $\QQ(f)$, for some modular function $f$ for which $J(f)$ coincides with the modular $j$-function.  Using our work in \S\ref{S:hauptmoduls}, we can construct an explicit modular function $h$ such that $\QQ(\zeta_N)(X_G)=\QQ(\zeta_N)(h)$, along with a rational function $J'(t)\in \QQ(\zeta_N)(t)$ such that $J'(h)=j$.    The function $f$ must satisfy $f=\psi(h)$ for some $\psi(t)\in \QQ(\zeta_N)(t)$ of degree $1$, and therefore $J'(h)=j=J(f)=J(\psi(h))$; this in turn implies that $J'(t)=J(\psi(t))$.  We then directly test all the modular functions $f:=\psi(h)$, where $\psi(t)\in \QQ(\zeta_N)(t)$ is one of the finitely many degree $1$ rational functions that satisfy $J'(t)=J(\psi(t))$.

In \S\ref{S:genus1}, we prove the part of Theorem~\ref{T:main} concerning genus $1$ groups.  Let $G$ be one of the genus~$1$ groups from \S\ref{S:group theory}.  One can show that $X_G$ has good reduction at all primes $p\nmid N$ and its modular interpretation gives a way to compute $\#X_G(\FF_p)$ directly from the group $G$, without requiring an explicit model.  By computing $\#X_G(\FF_p)$ for enough primes $p\nmid N$, one can determine the Jacobian $J_G$ of $X_G$ up to isogeny.  This allows us to compute the rank of $J_G(\QQ)$ which is an isogeny invariant of $J_G$.   We need only consider groups for which $J_G(\QQ)$ has positive rank since otherwise $X_G(\QQ)$ is finite; this leaves the $28$ genus $1$ groups in Theorem~\ref{T:main}.  These $28$ groups $G$ of genus~$1$ and a description of their morphisms $\pi_G$ already appear in the literature; our contribution lies in proving that there are no others.

In \S\ref{S:ell-adic proof}, we complete the proof of Theorem~\ref{T:ell-adic}, and in \S\ref{S:how} we explain how we found the rational functions $J(t) \in \QQ(t)$ whose verification is described in \S\ref{S:genus 0 check}.

Appendix~\ref{S:appendix} lists the 248 groups $G$ that appear in Theorem~\ref{T:main}, along with explicit maps from $X_G$ to the $j$-line; for the 220 groups of genus 0 these are rational functions $J(t)$, and for the~28 groups of genus~1 these are morphisms $J(x,y)$ from an explicit Weierstrass model for $X_G$ as an elliptic curve of positive rank.
One can use these maps to explicitly construct infinite families of elliptic curves $E/\QQ$ with distinct $j$-invariants whose $\ell$-adic Galois images match the groups $G$ listed in Theorem~\ref{T:ell-adic} and the groups $H$ listed in Corollary~\ref{C:ell-adic} by choosing appropriate quadratic twists.

\subsection{Related results}

Contemporaneous with our work, Rouse and Zureick-Brown independently computed explicit models for all modular curves $X_G/\QQ$ of 2-power level that have a non-cuspidal rational point, including all those for which $X_G(\QQ)$ is infinite \cite{RZB}.  The $X_G$ of 2-power level in our list agree with theirs, although we generally obtain different (but isomorphic) models (N.B. our groups are transposed relative to theirs; in our choice of the isomorphism $\Aut(E[N])\simeq\GL_2(\ZZ/N\ZZ)$ we view matrices in $\GL_2(\ZZ/N\ZZ)$ as acting on the left, rather than the right).

\subsection*{Acknowledgements}
We thank Jeremy Rouse and David Zureick-Brown for the feedback on an early draft of this article, and the referees for their careful review and helpful comments.

\subsection*{Notation and Terminology}
For each integer $n\geq 1$, we denote by $\zeta_n$ the $n$th root of unity $e^{2\pi i /n}$ in $\CC$, and let $K_n:=\QQ(\zeta_n)$ denote the corresponding cyclotomic field.  For any non-constant function $f\in K(t)$, where $K$ is a field, the \defi{degree} of $f$ is its degree as a morphism $\PP^1_K\to\PP^1_K$.   

For any ring $R$, we denote by $\Mat_2(R)$ the ring of $2\times 2$ matrices with coefficients in $R$.
We denote by $\Zhat$ the profinite completion of $\ZZ$, and view the profinite group
\[
\GL_2(\Zhat)\simeq \varprojlim_N \GL_2(\ZZ/N\ZZ)\simeq {\prod_\ell} \GL_2(\ZZ_\ell)
\]
as a topological group in the profinite topology.
If $G$ is an open subgroup of $\GL_2(\Zhat)$, we define its \defi{level} to be the least positive integer $N$ for which $G$ is the inverse image of a subgroup of $\GL_2(\ZZ/N\ZZ)$ under the natural projection $\GL_2(\Zhat)\to \GL_2(\ZZ/N\ZZ)$.
If $G$ is a subgroup of $\GL_2(\ZZ/N\ZZ)$, its level is defined to be the level of its inverse image in $\GL_2(\Zhat)$, which is necessarily a divisor of $N$.  For convenience we may identify the level $N$ subgroups of $\GL_2(\ZZ/N\ZZ)$ with their inverse images in $\GL_2(\Zhat)$, and conversely.  By the \defi{genus} of an open subgroup $G$ of $\GL_2(\Zhat)$ satisfying $-I\in G$ and $\det(G)=\Zhat^\times$, we mean the genus of the modular curve $X_G$ defined in \S\ref{S:modular}.  

For sets $S$ and $T$ we use $S-T$ to denote the set of elements that lie in $S$ but not $T$.

\section{Modular functions and modular curves}   \label{S:modular}

In this section, we summarize the background we need concerning modular curves.   

\subsection{Congruence subgroups}

Fix a congruence subgroup $\Gamma$ of $\SL_2(\ZZ)$, i.e., a subgroup of $\SL_2(\ZZ)$ containing 
\[
\Gamma(N):=\{A\in \SL_2(\ZZ) : A\equiv I \pmod{N}\}
\]
for some integer $N\geq 1$.   The smallest such $N$ is the \defi{level} of $\Gamma$.

The group $\Gamma$ acts on the complex upper half plane $\HH$ by linear fractional transformations, and the quotient $Y_\Gamma = \Gamma\backslash \HH$ is a smooth Riemann surface.  By adding \defi{cusps}, we can extend $Y_\Gamma$ to a smooth compact Riemann surface $X_\Gamma$.  We denote by $X(N)$ the Riemann surface $X_{\Gamma(N)}$.    The \defi{genus} of $\Gamma$ is the genus of the Riemann surface $X_\Gamma$.  
 
\subsection{Cusps}  
 
Define the extended upper half plane by $\HH^*:=\HH \cup \PP^1(\QQ) = \HH \cup \QQ \cup \{\infty\}$.
The action of $\Gamma$ extends to $\HH^*$ and we can identify the quotient $\Gamma\backslash \HH^*$ with $X_\Gamma$.
In particular, the cusps correspond to the $\Gamma$-orbits of $\QQ\cup \{\infty\}$.

\begin{lemma} \label{L:cusp orbits}
Let $a/b$ and $\alpha/\beta$ be elements of $\QQ\cup \{\infty\}$ with $\gcd(a,b)=\gcd(\alpha,\beta)=1$ (where we take $\infty=\pm 1/0$).   Then $\Gamma \cdot a/b = \Gamma \cdot \alpha/\beta$ if and only if  $\gamma  \big(\begin{smallmatrix} a \\ b\end{smallmatrix}\big) \equiv \pm  \big(\begin{smallmatrix} \alpha \\ \beta \end{smallmatrix}\big) \pmod{N}$ 
for some $\gamma \in \Gamma$.
\end{lemma}
\begin{proof}
For the case $\Gamma=\Gamma(N)$, see \cite{MR1291394}*{Lemma~1.42}.  The general case follows easily.
\end{proof}

Let $\pm \Gamma$ be the congruence subgroup generated by $-I$ and $\Gamma$.    From Lemma~\ref{L:cusp orbits}, we find that the cusps of $X_\Gamma$ correspond with the orbits of $\pm \Gamma$ on the set of $\big(\begin{smallmatrix} a \\ b\end{smallmatrix}\big) \in (\ZZ/N\ZZ)^2$ of order $N$.  Using this, it is straightforward to find representatives of the cusps of $X_\Gamma$.

\subsection{Modular functions}

A \defi{modular function} for $\Gamma$ is a meromorphic function of $X_\Gamma$; they correspond to meromorphic functions $f$ of $\HH$ that satisfy $f(\gamma \tau)=f(\tau)$ for all $\gamma \in \Gamma$ and are meromorphic at the cusps.  The function field $\CC(X_\Gamma)$ of $X_\Gamma$ consists of the meromorphic functions of $X_\Gamma$.

Let $\tau$ be a variable of the upper half plane.
Let $w$ be the width of the cusp at $\infty$, i.e., the smallest positive integer for which $\big(\begin{smallmatrix} 1 & w \\ 0 & 1 \end{smallmatrix}\big)$ is an element of $\Gamma$; it is a divisor of $N$.
For any rational number $m$,  define $q^{m} := e^{2\pi i m \tau}$.
Then any modular function $f$ for $\Gamma$ has a unique $q$-expansion
\[
f(\tau) = \sum_{n \in \ZZ } c_n q^{n/w},
\]
where the $c_n$ are complex numbers that are $0$ for all but finitely many $n<0$.   We will often refer to the $c_n$ as the \defi{coefficients }of $f$.

\subsection{Field of modular functions}\label{S:modular functions}
Fix a positive integer $N$.  Denote by $\calF_N$ the field of meromorphic functions of the Riemann surface $X(N)$ whose $q$-expansions have coefficients in $K_N:=\QQ(\zeta_N)$.
For example, $\calF_1=\QQ(j)$, where $j$ is the modular $j$-invariant.

For $f\in \calF_N$ and $\gamma\in \SL_2(\ZZ)$, let $f|_{\gamma} \in \calF_N$ denote the modular function satisfying $f|_{\gamma}(\tau)=f(\gamma\tau)$.

For each $d\in (\ZZ/N\ZZ)^\times$, let $\sigma_d$ be the automorphism of $K_N$ satisfying $\sigma_d(\zeta_N)=\zeta_N^d$.
We extend $\sigma_d$ to an automorphism of the field $\calF_N$ by defining
\[
\sigma_d(f) := \sum_n \sigma_d(c_n) q^{n/N},
\]
where $f$ has expansion $\sum_n c_n q^{n/N}$.  We now recall some facts about the extension $\calF_N$ of $\calF_1=\QQ(j)$.

\begin{proposition} \label{P:FN Galois}
The extension $\calF_N$ of $\QQ(j)$ is Galois.   There is a unique isomorphism
\[
\theta_N \colon \GL_2(\ZZ/N\ZZ)/\{\pm I\} \xrightarrow{\sim} \Gal(\calF_N/\QQ(j))
\]
such that the following hold for all $f\in \calF_N$:
\begin{alphenum}
\item \label{P:FN Galois a}
For $g\in \SL_2(\ZZ/N\ZZ)$, we have $\theta_N(g) f = f|_{\gamma^t}$, where $\gamma$ is any matrix in $\SL_2(\ZZ)$ that is congruent to $g$ modulo $N$ and $\gamma^t$ is the transpose of $\gamma$.
\item
For $g=\big(\begin{smallmatrix} 1 & 0  \\  0 & d\end{smallmatrix}\big) \in \GL_2(\ZZ/N\ZZ)$, we have $\theta_N(g) f = \sigma_d(f)$.
\end{alphenum}
Moreover, the algebraic closure of $\QQ$ in $\calF_N$ is $\QQ(\zeta_N)$; it corresponds to the subgroup $\SL_2(\ZZ/N\ZZ)/\{\pm I\}$.
\end{proposition}
\begin{proof}
This is well known; see \cite{MR648603}*{Ch.~2~\S2} for a summary (where the action given is a right action obtained as above but without the transpose in (\ref{P:FN Galois a})).
\end{proof}

Throughout the rest of the paper, we let $\GL_2(\ZZ/N\ZZ)$ act on $\calF_N$ via the homomorphism $\theta_N$ of Proposition~\ref{P:FN Galois}.
We set $g_*(f):=\theta_N(g)(f)$ for $g\in \GL_2(\ZZ/N\ZZ)$ and $f\in \calF_N$.

\begin{remark} 
There are other natural actions of $\GL_2(\ZZ/N\ZZ)$ on $\calF_N$; for example, one could replace $\gamma^t$ in condition (\ref{P:FN Galois a}) by $\gamma^{-1}$ or just act on the right.  Our choice is motivated by Proposition~\ref{P:XG rational} below.
\end{remark}

\subsection{Modular curves} \label{S:modular curves}

Let $G$ be a subgroup of $\GL_2(\ZZ/N\ZZ)$ satisfying $-I \in G$ and $\det(G)=(\ZZ/N\ZZ)^\times$.
Let $\calF_N^G$ be the subfield of $\calF_N$ fixed by the action of $G$ from Proposition~\ref{P:FN Galois}.
Proposition~\ref{P:FN Galois} and the assumption $\det(G)=(\ZZ/N\ZZ)^\times$ imply that $\QQ$ is algebraically closed in $\calF_N^G$. \\

The \defi{modular curve} $X_G$ associated with $G$ is the smooth projective curve with function field $\calF_N^G$.
The curve $X_G$ is defined over $\QQ$ and is geometrically irreducible.
The inclusion of fields $\calF_N^G \supseteq \calF_1=\QQ(j)$ gives rise to a non-constant morphism
\[
\pi_G\colon X_G \to \Spec \QQ[j] \cup \{\infty\} = \PP^1_\QQ
\]
of degree $[\GL_2(\ZZ/N\ZZ):G]$.  Moreover, given another group $G\subseteq G' \subseteq \GL_2(\ZZ/N\ZZ)$, the inclusion $\calF_N^{G'} \subseteq \calF_N^G$ induces a nonconstant morphism $X_{G}\to X_{G'}$ of degree $[G':G]$.  Composing $X_G\to X_{G'}$ with $\pi_{G'}$ gives the morphism $\pi_G$.

Let $\Gamma$ be the congruence subgroup consisting of $\gamma \in \SL_2(\ZZ)$ for which $\gamma^t$ modulo $N$ lies in $G\cap\SL_2(\ZZ/N\ZZ)$.    The level of $\Gamma$ divides, but need not equal, $N$.   

\begin{lemma} \label{L:modular connection}
\begin{romanenum}
\item \label{L:modular connection i}
The field $K_N(X_G)$, i.e., the function field of the base extension of $X_G$ to $K_N$, is the field consisting of $f\in \calF_N$ satisfying $f|_\gamma=f$ for all $\gamma\in \Gamma$.

\item \label{L:modular connection ii}
The genus of the modular curve $X_G$ is equal to the genus of $\Gamma$.
\end{romanenum}
\end{lemma}
\begin{proof}
Proposition~\ref{P:FN Galois} implies that $K_N$ is algebraically closed in $\calF_N$ and that we have an isomorphism $\Gal(\calF_N/K_N(j))\xrightarrow{\sim} \SL_2(\ZZ/N\ZZ)/\{\pm I\}$.   Thus $K_N(X_G)$ is the subfield of $\calF_N$ fixed by $G\cap \SL_2(\ZZ/N\ZZ)$.    Part (\ref{L:modular connection i}) is now clear.

Since $K_N$ is algebraically closed in $\calF_N$ and $\QQ$ is algebraically closed in $\QQ(X_G)$, we have 
\[
[\CC\cdot K_N(X_G): \CC(j)]=[K_N(X_G):K_N(j)]=[\QQ(X_G):\QQ(j)]=[\GL_2(\ZZ/N\ZZ):G].
\]
Since $\det(G)=(\ZZ/N\ZZ)^\times$, we deduce that $[\CC\cdot K_N(X_G): \CC(j)]=[\SL_2(\ZZ):\Gamma]$.

Clearly each $f\in K_N(X_G)$ is a modular function for $\Gamma$, thus $\CC\cdot K_N(X_G) \subseteq \CC(X_\Gamma)$.  We in fact have $\CC\cdot K_N(X_G)=\CC(X_\Gamma)$, since $[\CC\cdot K_N(X_G): \CC(j)]=[\SL_2(\ZZ):\Gamma]=[\CC(X_\Gamma):\CC(j)]$.   The curve $X_G$ has the same genus as the Riemann surface $X_\Gamma$ because $\CC(X_G)=\CC(X_\Gamma)$.
\end{proof}

\begin{remark} \label{R:transpose business}
Another natural congruence subgroup to study is the congruence subgroup $\Gamma'$  consisting of $\gamma\in \SL_2(\ZZ)$ such that $\gamma$ modulo $N$ lies in $G\cap\SL_2(\ZZ/N\ZZ)$, which we use later in the paper.    Observe that the congruence subgroups $\Gamma$ and $\Gamma'$ are conjugate in $\SL_2(\ZZ)$; indeed, we have $B^{-1} \gamma B = (\gamma^t)^{-1}$ for all $\gamma\in \SL_2(\ZZ)$, where $B:=\big(\begin{smallmatrix} 0 & 1 \\ -1 & 0\end{smallmatrix}\big)$.   Thus $\Gamma$ and $\Gamma'$ have the same genus.
\end{remark}

The following proposition is crucial to our application.

\begin{proposition} \label{P:XG rational}
Let $E$ be an elliptic curve defined over $\QQ$ with $j_E \notin \{0,1728\}$.
Then $\rho_{E,N}(\Gal_\QQ)$  is conjugate in $\GL_2(\ZZ/N\ZZ)$ to a subgroup of $G$ if and only if $j_E$ belongs to $\pi_G(X_G(\QQ))$.
\end{proposition}
\begin{proof}
See \S3 of \cite{PossibleImages} for a proof.
\end{proof}

\subsection{Modular curves and open subgroups} \label{SS:modular curves}

Fix an open subgroup $G$ of $\GL_2(\Zhat)$ that satisfies $-I\in G$ and $\det(G)=\Zhat^\times$.    Let $N\geq 1$ be an integer that is divisible by the level of $G$.  Define the modular curve
\[
X_G:=X_{\bbar G},
\]
where $\bbar G\subseteq \GL_2(\ZZ/N\ZZ)$ is the image of $G$ modulo $N$.  Observe that the modular curve $X_G$ and its function field do not depend on the initial choice of $N$.

Every (open) subgroup $G'$ of $\GL_2(\Zhat)$ that contains $G$ satisfies $-I\in G'$ and $\det(G')=\Zhat^\times$, and we have a morphism $X_{G}\to X_{G'}$.  With $G'=\GL_2(\Zhat)$, we obtain a morphism $\pi_G\colon X_{G}\to X_{G'}=\PP^1_\QQ$ to the $j$-line that agrees with $\pi_{\bbar G}$.  The following is equivalent to Proposition~\ref{P:XG rational}.

\begin{proposition} \label{P:XG rational 2}
Let $E$ be an elliptic curve defined over $\QQ$ with $j_E \notin \{0,1728\}$.  Then $\rho_{E}(\Gal_\QQ)$ is conjugate in $\GL_2(\Zhat)$ to a subgroup of $G$ if and only if $j_E$ belongs to $\pi_G(X_G(\QQ))$.\qed
\end{proposition}

\subsection{Complex conjugation} \label{SS:complex conj}

Fix a subgroup $G$ of $\GL_2(\ZZ/N\ZZ)$ satisfying $-I\in G$ and $\det(G)=(\ZZ/N\ZZ)^\times$.   For our curve $X_G$ to have rational points, we need $G$ to contain an element that ``looks like'' complex conjugation.

\begin{lemma} \label{L:complex conjugation}
For any elliptic curve $E/\QQ$ and integer $N>1$, the group $\rho_{E,N}(\Gal_\QQ)$ contains an element that is conjugate in $\GL_2(\ZZ/N\ZZ)$ to $\smallmat{1}{0}{0}{-1}$ or $\smallmat{1}{1}{0}{-1}$. 
\end{lemma}
\begin{proof}
This follows from Proposition~3.5 of \cite{PossibleIndices} (and its proof for the cases $j_E \in \{0,1728\}$).
\end{proof}

Note that $\smallmat{1}{0}{0}{-1}$ and $\smallmat{1}{1}{0}{-1}$ are conjugate to each other in $\GL_2(\ZZ/N\ZZ)$ if $N$ is odd.   If $G$ does not contain an element that is conjugate in $\GL_2(\ZZ/N\ZZ)$ to $\smallmat{1}{0}{0}{-1}$ or $\smallmat{1}{1}{0}{-1}$, then $X_G(\QQ)$ must be empty since $X_G(\RR)$ is finite (by \cite{PossibleIndices}*{Proposition~3.5}), hence empty, since $X_G$ is non-singular.

\section{Group theoretic computations} \label{S:group theory}

We define an \defi{admissible group} to be an open subgroup $G$  of $\GL_2(\Zhat)$ for which the following conditions hold:
\begin{itemize}
\item
$G$ has prime power level,
\item
$-I \in G$ and $\det(G)=\Zhat^\times$,
\item
$G$ contains an element that is conjugate in $\GL_2(\Zhat)$ to $\smallmat{1}{0}{0}{-1}$ or $\smallmat{1}{1}{0}{-1}$. 
\end{itemize}

The condition $\det(G)=\Zhat^\times$ is needed for Proposition~\ref{P:XG rational 2} since we have $\det(\rho_{E}(\Gal_\QQ))=\Zhat^\times$.   If we were interested in elliptic curve defined over other number fields, then we could loosen this restriction which could increase the base field of the modular curve $X_G$.    

The condition $-I \in G$ is also needed in Proposition~\ref{P:XG rational 2}.   For an elliptic curve $E/\QQ$, there is a quadratic twist $E'/\QQ$, which automatically has the same $j$-invariant as $E$, such that $-I \in \rho_{E}(\Gal_\QQ)$.  

The last condition on $G$ is necessary in order for $X_G(\QQ)$ to be non-empty, as explained in \S\ref{SS:complex conj}.

\begin{proposition}
Let $G$ be an admissible group of genus $0$.  The set $X_G(\QQ)$ is infinite.
\end{proposition}
\begin{proof}
We have $X_G(\RR)\ne\emptyset$ by Proposition 3.5 of \cite{PossibleIndices}.
For primes $p$ not dividing its prime power level the modular curve $X_G$ has good reduction at $p$ and $X_G(\QQ_p)\ne\emptyset$, since the reduction of $X_G$ to $\FF_p$ necessarily has rational points that can be lifted to $\QQ_p$ via Hensel's lemma.  Thus $X_G$ has rational points locally at all but at most one place of $\QQ$.
The product formula for Hilbert symbols and the Hasse-Minkowksi Theorem then imply that $X_G$ has a rational point and is thus isomorphic to~$\PP^1$ and has infinitely many rational points.
\end{proof}

\begin{remark}
As shown by Proposition 3.1, our three criteria for admissibility rule out genus 0 curves with no rational points.  There are ten groups $G$ of $2$-power level that satisfy our first two criteria but not the third; these give rise to the ten pointless conics $X_G$ found in \cite{RZB}.
There are three such groups of $3$-power level, three of $5$-power level, and none of higher prime-power level.
\end{remark}

Fix an integer $g\geq 0$.  In this section, we explain how to enumerate all admissible subgroups $G$ of $\GL_2(\Zhat)$, up to conjugacy, that have genus at most $g$.  We shall apply these methods with $g=1$ to verify Theorem~\ref{T:groups 0 and 1} below, and to find explicit representatives of these conjugacy classes of groups; \texttt{Magma} \cite{Magma} scripts that perform this enumeration can be found at \cite{MagmaScripts}.  

\begin{theorem}\label{T:groups 0 and 1}
\begin{romanenum}

\item \label{T:groups 0 and 1 i}
Up to conjugacy in $\GL_2(\Zhat)$, there are $220$ admissible subgroups of genus $0$.

\item \label{T:groups 0 and 1 ii}
Up to conjugacy in $\GL_2(\Zhat)$, there are $250$ admissible subgroups of genus $1$.
\end{romanenum}
\end{theorem}

\begin{remark}
The $220$ admissible subgroups $G$ of genus $0$, up to conjugacy, are precisely those given in Tables \ref{table:g0three}--\ref{table:g0two}. More precisely, for each entry of the table, we have an integer $N$ and a set of generators that generates the image in $\GL_2(\ZZ/N\ZZ)$ of an admissible group of level $N$ and genus $0$.
\end{remark}

\begin{remark}
The $28$ admissible subgroups $G$ of genus $1$ that have infinitely many rational points, up to conjugacy, are precisely those given in Table \ref{table:g1}, of which 27 have level 16 and  1 has level 11.  The levels arising among the remaining 222 are 7, 8, 9, 11, 16, 17, 19, 27, 32, and 49.
\end{remark}

For a fixed admissible group $G$ of level $N$, let $\Gamma$ be the congruence subgroup of $\SL_2(\ZZ)$ consisting of matrices whose image modulo $N$ lies in the image of $G$ modulo $N$; the level of $\Gamma$ necessarily divides $N$, and $\Gamma$ contains $-I$.   By Lemma~\ref{L:modular connection}(\ref{L:modular connection ii}) and Remark~\ref{R:transpose business}, the modular curve $X_G$ has the same genus as $\Gamma$. 

The basic idea of our computation is to reverse the process above; we start with a congruence subgroup $\Gamma$ of genus at most $g$ and prime power level, and then enumerate the possible groups $G$ that could produce $\Gamma$.

Let $S_g$ be the set of congruences subgroups of $\SL_2(\ZZ)$ of prime power level that contain $-I$ and have genus at most $g$.    We know that the set $S_g$ is finite from a theorem of Dennin \cite{MR0342466}.  When $g\leq 24$, and in particular, for $g=1$, we can explicitly determine the elements of $S_g$ from the tables of Cummins and Pauli \cite{MR2016709}, which are available online at \cite{CPwebsite} (their methods can also be extended to larger $g$).

Let $L_g$ be the set of primes that divide the level of some congruence subgroup $\Gamma \in S_g$.   The set $L_g$ is finite, since $S_g$ is finite, and we have $L_1=\{2,3,5,7,11,13,17,19\}$.  If $G$ is an admissible group of genus at most $g$, then its level must be a power of a prime $\ell \in L_g$.   For the rest of the section, we fix a prime $\ell \in L_g$.  Since $L_g$ is finite, it suffices to explain how to compute the admissible groups $G$ with genus at most $g$ whose level is a power of $\ell$, and we need only consider levels strictly greater than $1$ since $\GL_2(\Zhat)$ is the only admissible group of level $1$.\medskip

Fix a prime power $N:=\ell^n> 1$, and consider any congruence subgroup $\Gamma\in S_g$ whose level divides~$N$.
By enumerating subgroups of $\GL_2(\ZZ/N\ZZ)$ one can explicitly determine those subgroups $G_N$ that satisfy the following conditions:
\begin{enumerate}
\item  \label{E:i}
$G_N$ has level $N$,
\item \label{E:ii}
$G_N \cap \SL_2(\ZZ/N\ZZ)$ is equal to the image of $\Gamma$ modulo $N$,
\item \label{E:iii}
$\det(G_N)=(\ZZ/N\ZZ)^\times$,
\item \label{E:iv}
$G_N$ contains an element that is conjugate in $\GL_2(\ZZ/N\ZZ)$ to $\smallmat{1}{0}{0}{-1}$ or $\smallmat{1}{1}{0}{-1}$.
\end{enumerate}

Let $H$ be the image of $\Gamma$ in $\SL_2(\ZZ/N\ZZ)$.   
The group $H=G_N\cap \SL_2(\ZZ/N\ZZ)$ is normal in $G_N$ and hence $G_N$ is a subgroup of the normalizer $K$ of $H$ in $\GL_2(\ZZ/N\ZZ)$.  So rather than searching for $G_N$ in $K$, we can work in the quotient $K/H$ where the image of $G_N$ is an abelian group isomorphic to $(\ZZ/N\ZZ)^\times$.  
Using \texttt{Magma}, we can efficiently enumerate all abelian subgroups~$A$ of $K/H$ of order $\#(\ZZ/N\ZZ)^\times$.
For each such subgroup $A$ we then test whether its inverse image $G_N$ in $K$ satisfies conditions (\ref{E:i})--(\ref{E:iv}) above. 

Let $G$ be the subgroup of $\GL_2(\Zhat)$ consisting of those matrices whose image modulo $N$ lies in a fixed group $G_N$ satisfying the conditions (\ref{E:i})--(\ref{E:iv}).  The group $G$ is admissible of level $N$ and has genus at most $g$. 
Moreover, it is clear that every admissible group of level $N$ and genus at most $g$ arises in this manner.
 
Fix an integer $e\geq 1$.  By applying the above method with $1\leq n \leq e$, we obtain all admissible groups $G$ of genus at most $g$ and level dividing $\ell^e$.
Our algorithm proceeds by applying this procedure to increasing values of $e$.
In order for it to terminate we need to know that there are only finitely many admissible groups $G$ of $\ell$-power level and genus at most $g$, and we need an explicit way to determine when we have reached an $e$ that is large enough to guarantee that we have found them all.
Proposition~\ref{P:valid algorithm} below addresses both issues.

\begin{proposition} \label{P:valid algorithm}
\begin{romanenum}
\item \label{P:valid algorithm i}
There are only finitely many admissible groups $G$ with genus at most $g$ whose level is a power of $\ell$.
\item \label{P:valid algorithm ii}
Take any integer $n\geq 2$ with $n\neq 2$ if $\ell=2$.  Define $N:=\ell^n$.  Suppose that there is no subgroup $G_N$ of $\GL_2(\ZZ/N\ZZ)$ that satisfies conditions (\ref{E:i})--(\ref{E:iv}) for some $\Gamma\in S_g$ with level dividing $N$.   Then any admissible group $G$ of genus at most $g$ with level a power of $\ell$ has level at most $N$.
\end{romanenum}
\end{proposition}

The remainder of this section is devoted to proving Proposition~\ref{P:valid algorithm}.  We will need the following basic lemma.

\begin{lemma}\label{L:index bound}
Let $\ell$ be a prime and let $G$ be an open subgroup of $\GL_2(\ZZ_\ell)$.   For each integer $m\ge 1$, let $i_m$ be the index of the image of $G$ in $\GL_2(\ZZ/\ell^m\ZZ)$.
If $i_{n+1}=i_n$ for an integer $n\ge 1$, with $n\neq 1$ if $\ell=2$, then $[\GL_2(\ZZ_\ell):G]=i_n$.
\end{lemma}
\begin{proof}
Since $G$ is an open subgroup, it suffices to prove $i_{m+1}=i_m$ for all $m\ge n$; we proceed by induction on $m$.
The base case is given, so we assume $i_{m+1}=i_m$ for some $m\ge n$; we need to show that $i_{m+2}=i_{m+1}$.
Let $G_m$ denote the image of $G$ in $\GL_2(\ZZ/\ell^m\ZZ)$.
Reduction modulo $\ell^m$ gives exact sequences related by inclusions
\begin{center}
\begin{tikzcd}
1\arrow{r} & K_{m+1}\arrow{r} & \GL_2(\ZZ/\ell^{m+1}\ZZ)\arrow{r} & \GL_2(\ZZ/\ell^m\ZZ)\arrow{r} & 1\\
1\arrow{r} & H_{m+1} \arrow{r}\arrow[hookrightarrow]{u} & G_{m+1}\arrow{r}\arrow[hookrightarrow]{u} & G_m\arrow{r}\arrow[hookrightarrow]{u} & 1.
\end{tikzcd}
\end{center}
The inductive hypothesis $i_{m+1}=i_m$ implies that the kernels $H_{m+1}$ and $K_{m+1}$ coincide; in particular, $H_{m+1}$ is as large as possible (i.e., it has order $\ell^4$).
It  thus suffices to show that the kernel $H_{m+2}$ of the reduction map from $G_{m+2}$ to $G_{m+1}$ also has order $\ell^4$.  We have $|H_{m+2}|\leq \ell^4$, so it suffices to give an injective map $H_{m+1}\to H_{m+2}$.

Let $M$ be an element of $G$ whose image in $G_{m+1}$ lies in $H_{m+1}$;
then $M=I+\ell^m A$ for some $A\in \Mat_2(\ZZ_\ell)$.
Since $m\ge 1$, with $m\ge 2$ if $\ell=2$, we have
\[
(1+\ell^m A)^\ell = 1+\tbinom{\ell}{1}\ell^mA+\tbinom{\ell}{2}\ell^{2m}A^2+\cdots\equiv 1+\ell^{m+1}A\pmod{\ell^{m+2}}.
\]
The $\ell$-power map thus induces an injection $H_{m+1}\to H_{m+2}$.
\end{proof}

\begin{remark}
Lemma~\ref{L:index bound} holds more generally.  One can replace $\GL_2(\ZZ_\ell)$ with the unit group of any (unital associative) $\ZZ_\ell$-algebra $\calA$ that is torsion-free and finitely-generated as a $\ZZ_\ell$-module (in the lemma, $\calA=\Mat_2(\ZZ_\ell)$); the proof is exactly the same.
\end{remark}

\begin{proof}[Proof of Proposition~\ref{P:valid algorithm}(\ref{P:valid algorithm i})]
Let $\calG$ be the set of admissible groups of genus at most $g$ whose level is a power of $\ell$.
Note that if $G'$ is a subgroup of $\GL_2(\Zhat)$ containing some $G\in\calG$, then $G'\in\calG$.  We wish to show that $\calG$ is finite.

We claim that any admissible group $G$ has only finitely many maximal subgroups that are also admissible and whose level is a power of $\ell$.   It suffices to show that an open subgroup $H$ of $\GL_2(\ZZ_\ell)$ has only finitely many open maximal subgroups.  Let $\Phi(H)$ be the Frattini subgroup of $H$; it is the intersection of the maximal closed proper subgroups of $H$.  By the proposition in \S10.5 of \cite{MR1757192}, $\Phi(H)$ is an open subgroup of $H$.   This proves the claim.

Now suppose that $\calG$ is infinite.
The claim implies that $\calG$ contains an infinite descending chain $G_1\supsetneq G_2\supsetneq G_3\supsetneq \cdots$ (let $G_1=\GL_2(\Zhat)\in\calG$, let $G_2\in\calG$ be one of the finitely many maximal subgroups of $G_1$ in $\calG$ that has infinitely many subgroups in $\calG$, and continue in this fashion).  For each $i\ge 1$, let $\Gamma_i$ be the congruence subgroup associated to $G_i$ (i.e., $\Gamma_i$ consists of the matrices in $\SL_2(\ZZ)$ whose image modulo $N$ lies in the image modulo $N$ of $G_i$, where $N$ is the level of $G_i$); then $\Gamma_i \in S_g$.   Since $[\GL_2(\Zhat):G_i]=[\SL_2(\ZZ):\Gamma_i]$, we have inclusions $\Gamma_1\supsetneq \Gamma_2\supsetneq \Gamma_3\supsetneq \cdots$.  This contradicts the finiteness of $S_g$ and the proposition follows.
\end{proof}

\begin{proof}[Proof of Proposition~\ref{P:valid algorithm}(\ref{P:valid algorithm ii})]
Fix an integer $n\geq 1$ as in the statement of part (\ref{P:valid algorithm ii}).  Suppose there is an integer $m>n$ such that there is an admissible group $G$ of level $\ell^m$ and genus at most $g$. 

With $N:=\ell^n$,  let $G_N$ be the image of $G$ in $\GL_2(\ZZ/N\ZZ)$.   The curve $X_{G_N}$ has genus at most $g$ since it is dominated by $X_G$.   Therefore, conditions (\ref{E:ii}), (\ref{E:iii}) and (\ref{E:iv}) hold for some $\Gamma\in S_g$ with level dividing $N$.  Our assumption on $n$ implies that the level of $G_N$ is a proper divisor of $N$.   This implies that the index $i_{n}:=[\GL_2(\ZZ/N\ZZ):G_N]$ agrees with $i_{n-1}:=[\GL_2(\ZZ/\ell^{n-1}\ZZ):G_{\ell^{n-1}}]$, where $G_{\ell^{n-1}}$ is the image of $G$ in $\GL_2(\ZZ/\ell^{n-1}\ZZ)$.   Since $i_n=i_{n-1}$, Lemma~\ref{L:index bound} implies that $[\GL_2(\ZZ_\ell):G]=i_{n-1}$.   However, this means that $G$ has level dividing $\ell^{n-1}$ which is impossible since, by assumption, $G$ has level $\ell^m>\ell^{n-1}$.  Therefore, no such admissible group $G$ exists.
\end{proof}

\section{Construction of hauptmoduls} \label{S:hauptmoduls}

Fix a congruence subgroup $\Gamma$ of genus $0$ and level $N$.  The function field of $X_\Gamma$ is then of the form $\CC(h)$, where the function $h\colon X_\Gamma \to \CC\cup \{\infty\}$ gives an isomorphism between $X_\Gamma$ and the Riemann sphere; in particular, $h$ has a unique (simple) pole.

We may choose $h$ so that its unique pole is at the cusp $\infty$; we will call such an $h$ a \defi{hauptmodul} of $\Gamma$.
Every hauptmodul of $\Gamma$ is then of the form $a h +b$ for some complex numbers $a\neq 0$ and $b$.   
For example, the familiar modular $j$-invariant
\[
j(\tau)=q^{-1} + 744 + 196884q + 21493760q^2 + 864299970q^3+\cdots
\]
is a hauptmodul for $\SL_2(\ZZ)$.  If $h$ is a hauptmodul for $\Gamma$, then we have an inclusion of function fields $\CC(j) \subseteq \CC(h)$ and hence $J(h)=j$ for a unique rational function $J(h)\in \CC(t)$.\medskip

The main task of \S\ref{S:hauptmoduls} is to describe how to find an \emph{explicit} hauptmodul $h$ of $\Gamma$ in terms of Siegel functions when $N$ is a prime power.  Our $h$ will have coefficients in $K_N$.  In \S\ref{SS:hauptmoduls J}, we explain how to compute the rational function $J(t)$ corresponding to $h$.

\subsection{Siegel functions} \label{SS:Siegel}

Take any pair $a=(a_1,a_2) \in \QQ^2-\ZZ^2$.
We define the \defi{Siegel function} $g_a(\tau)$ to be the holomorphic function $\HH\to \CC^\times$ defined by the series
\[
-q^{\frac{1}{2} B_2(a_1)}\cdot  e(a_2(a_1-1)/2)\cdot  
(1- e(a_2) q^{a_1}) \prod_{n=1}^\infty (1- e(a_2) q^{n+a_1})(1- e(-a_2) q^{n-a_1}),
\]
where $e(z)=e^{2\pi i z}$ and $B_2(x)=x^2-x+1/6$.    

Recall that the \defi{Dedekind eta function} is the holomorphic function on $\HH$ given by
\[
\eta(\tau):=q^{1/24} \prod_{n=1}^\infty(1-q^n).
\]
For each $\gamma=\big(\begin{smallmatrix} a & b\\ c & d\end{smallmatrix}\big)\in \SL_2(\ZZ)$, there is a unique $12$-th root of unity $\varepsilon(\gamma) \in \CC^\times$ such that 
\begin{equation} \label{E:Siegel basics 2}
\eta(\gamma \tau)^2 = \varepsilon(\gamma)(c\tau+d) \eta(\tau)^2.
\end{equation}
We can characterize the map $\varepsilon\colon \SL_2(\ZZ)\to \CC^\times$ by the property that it is a homomorphism satisfying $\varepsilon\big(\big(\begin{smallmatrix} 1 & 1 \\ 0 & 1\end{smallmatrix}\big)\big) = \zeta_{12}$ and $\varepsilon\big(\big(\begin{smallmatrix} 0 & 1 \\ -1 & 0\end{smallmatrix}\big)\big) = \zeta_{4}$, cf.~\cite{MR648603}*{Ch.~3~\S5}.
Moreover, the kernel of $\varepsilon$ is a congruence subgroup of level $12$ and agrees with the commutator subgroup of $\SL_2(\ZZ)$.

The following lemma gives several key properties of Siegel functions.

\begin{lemma}   \label{L:Siegel basics}
For any $\gamma\in \SL_2(\ZZ)$, $a\in \QQ^2-\ZZ^2$, and $b\in \ZZ^2$, the following hold:
\begin{romanenum}
\item \label{L:Siegel basics i}
$g_{-a}=-g_a$,
\item \label{L:Siegel basics ii}
$g_{a+b} = (-1)^{b_1+b_2+b_1b_2} \cdot e((b_2 a_1 - b_1 a_2)/2) \cdot g_{a}$,
\item  \label{L:Siegel basics iii}
$g_a |_\gamma= \varepsilon(\gamma) \cdot g_{a\gamma}$, where we view $a$ as a row vector.  
\end{romanenum}
\end{lemma}
\begin{proof}
In \cite{MR648603}*{Ch.~2~\S1}, we see that $g_{a}(\tau)=\mathfrak{k}_{a}(\tau) \eta(\tau)^2$, where $\mathfrak{k}_{a}(\tau)$ is a Klein form (with $W=W_\tau$ in the notation of loc.~cit.).
Part (\ref{L:Siegel basics ii}) follows directly from property K2 of \S1 of \cite{MR648603}*{Ch.~2}.

Take any $\gamma\in \SL_2(\ZZ)$ and let $(c,d)$ be the last row of $\gamma$.
From properties K0 and K1 of \S1 of \cite{MR648603}*{Ch.~2}, we find that 
\begin{equation} \label{E:Siegel basics 1}
\mathfrak{k}_a(\gamma \tau)=(c\tau +d)^{-1} \mathfrak{k}_{a\gamma}(\tau)
\end{equation}
From (\ref{E:Siegel basics 2}) and (\ref{E:Siegel basics 1}), we deduce that $g_a(\gamma \tau) = \varepsilon(\gamma)\cdot g_{a\gamma}(\tau)$ which proves part (\ref{L:Siegel basics iii}).
Finally, part (\ref{L:Siegel basics i}) follows from part (\ref{L:Siegel basics iii}) with $\gamma=-I$, since $\varepsilon(-I)=-1$. 
\end{proof}

For an integer $N>1$, let $\calA_N$ be the set of pairs $(a_1,a_2) \in N^{-1}\ZZ^2-\ZZ^2$ that satisfy one of the following conditions:  
\begin{itemize}
\item $0< a_1 < 1/2$ and $0 \leq a_2 <1$, 
\item $a_1=0$ and $0<a_2 \leq 1/2$, 
\item $a_1=1/2$ and $0 \leq a_2 \leq 1/2$.
\end{itemize}
The set $\calA_N$ is chosen so that every non-zero coset of $(N^{-1} \ZZ^2)/\ZZ^2$ is represented by an element of the form $a$ or $-a$ for a unique $a\in \calA_N$.
So for any $a\in N^{-1} \ZZ^2 - \ZZ^2$, we can use parts (\ref{L:Siegel basics i}) and (\ref{L:Siegel basics ii}) of Lemma~\ref{L:Siegel basics} to show that
\[
g_a = \epsilon \cdot \zeta \cdot g_{a'}
\]
for an explicit sign $\epsilon \in \{\pm 1\}$, $N$-th root of unity $\zeta$, and pair $a'\in \calA_N$.

\subsection{Siegel orbits}

Now fix a congruence subgroup $\Gamma$ of level $N>1$.
For each $a\in \calA_N$ and $\gamma \in \SL_2(\ZZ)$, let $a*\gamma$ be the unique element of $\calA_N$ such that $a*\gamma$ or $-a*\gamma$ lies in the coset $a\gamma+\ZZ^2$.
The map 
\[
\calA_N \times \SL_2(\ZZ)\to \calA_N, \quad (a,\gamma)\mapsto a* \gamma
\] 
then gives a right action of $\SL_2(\ZZ)$ on $\calA_N$.
In particular, this gives a right action of $\Gamma$ on $\calA_N$.  

Fix a $\Gamma$-orbit $\OO$ of $\calA_N$ and define
\[
g_\OO := {\prod}_{a\in\OO} \,g_a;
\]
it is a holomorphic function $\HH\to \CC^\times$.  

\begin{lemma} \label{L:power of g}
The function $g_\OO^{12N}$ is a modular function for $\Gamma$.    Every pole and zero of $g_\OO^{12N}$ on $X_\Gamma$ is a  cusp.
\end{lemma}
\begin{proof}
Take any $\gamma\in \Gamma$ and  $a\in \calA_N$.
By Lemma~\ref{L:Siegel basics}(\ref{L:Siegel basics iii}), we have $g_a^{12N}|_\gamma = g_{a\gamma}^{12N}$.
We have $a\gamma = \epsilon\cdot (a*\gamma +b)$ for some $\epsilon\in\{\pm 1\}$ and $b\in \ZZ^2$.
By parts (\ref{L:Siegel basics i}) and (\ref{L:Siegel basics ii}) of Lemma~\ref{L:Siegel basics}, we find that $g_a^{12N}|_\gamma=g_{a\gamma}^{12N}$ is equal to $g_{a*\gamma}^{12N}$.
Therefore,
\[
g_\OO^{12N}|_{\gamma} = \prod_{a\in \OO} g_a^{12N}|_\gamma = \prod_{a\in \OO} g_{a*\gamma}^{12N} = g_\OO^{12N},
\]
where the last equality uses the fact that the map $\OO\to \OO$, $a\mapsto a*\gamma$ is a bijection (since $\OO$ is a $\Gamma$-orbit).
The remaining statement about the poles and zeros of $g_\OO^{12N}$ follows immediately since each $g_a$ is holomorphic and non-zero on $\HH$.
\end{proof}

Let $P_1,\ldots, P_r$ be the cusps of $X_\Gamma$.
Choose a representative $s_j \in \QQ\cup\{\infty\}$ of each cusp $P_j$ and a matrix $A_j \in \SL_2(\ZZ)$ satisfying $A_j \cdot \infty = s_j$.
Let $w_j$ be the \defi{width} of the cusp $P_j$; it is the smallest positive integer $b$ such that $A_j \big(\begin{smallmatrix} 1 & b \\ 0 & 1\end{smallmatrix}\big) A_j^{-1}$ is an element of $\Gamma$.

For a non-zero meromorphic function $f$ of $\HH$ given by a $q$-expansion, we define $\ord_q (f)$ to be the smallest rational number $m$ such that there is a non-zero term of the form $q^{m}$ in the expansion of~$f$.
For each cusp $P_j$, define the map 
\[
v_{P_j}\colon \CC(X_\Gamma)^\times \to \ZZ,\quad f\mapsto w_j \cdot \ord_q(f|_{A_j});
\]
it is a surjective homomorphism and agrees with the valuation giving the order of vanishing of a function at $P_j$.
We extend $\ord_q$ and $v_{P_j}$ by setting $\ord_q(0)=+\infty$ and $v_{P_j}(0)=+\infty$.  

We now give a computable expression for the divisor of $g_\OO^{12N}$  on $X_{\Gamma}$.

\begin{lemma} \label{L:explicit D}
With notation as above, we have
\[
\operatorname{div}(g_\OO^{12N}) = \sum_{j=1}^r \bigg( 6N w_j \sum_{a\in \OO} B_2\big(\big\langle (aA_j)_1\big\rangle\big) \bigg)\cdot P_j,
\]
where $B_2(x)=x^2-x+1/6$, $(aA_j)_1$ is the first coordinate of the row vector $aA_j$, and $\langle x\rangle$ denotes the positive fractional part of the real number $x$, chosen so $0\leq \langle{x}\rangle < 1$ and  $x-\langle x\rangle \in \ZZ$.
\end{lemma}
\begin{proof} 
For any $a\in (N^{-1}\ZZ^2)-\ZZ^2$, we have $\ord_q(g_a) = \tfrac{1}{2}\cdot B_2(\langle a_1 \rangle)$, cf.~\cite{MR648603}*{p.~31}.   We have
\[
v_{P_j}(g_{\OO}^{12N}) = \sum_{a\in \OO} v_{P_j}(g_a^{12N}) = \sum_{a\in \OO} w_j \ord_q(g_a^{12N} |_{A_j}) = \sum_{a\in \OO} w_j \ord_q(g_{aA_j}^{12N}),
\]
where the last equality uses Lemma~\ref{L:Siegel basics}(\ref{L:Siegel basics iii}).    Therefore,
\[
v_{P_j}(g_{\OO}^{12N}) = \sum_{a\in \OO} 12N w_j \ord_q(g_{aA_j}) = 6N w_j \sum_{a\in \OO} B_2\big(\big\langle (aA_j)_1\big\rangle \big).
\]
Since the poles and zeros of $g_\OO^{12N}$ are all cusps, we have $\operatorname{div}(g_\OO^{12N}) = \sum_{i=1}^r v_{P_j}(g_\OO^{12N})\cdot P_j$, and the lemma follows immediately.
\end{proof}

\subsection{Constructing hauptmoduls of prime power level} \label{SS:hauptmodul construction}

Fix a congruence subgroup $\Gamma$ of $\SL_2(\ZZ)$ of \emph{prime power} level $N>1$ that has genus $0$.   
Let $P_1,\ldots, P_r$ be the cusps of $\Gamma$; we choose our cusps so that $P_1$ is the cusp at $\infty$.

In this section, we explain how to construct an explicit hauptmodul of $\Gamma$ whose $q$-expansion has coefficients in $K_N$.
Moreover, our hauptmodul will be of the form
\begin{equation} \label{E:hauptmodul sum}
\sum_{i=1}^M \zeta_{2N^2}^{e_i} \prod_{a\in \calA_N} g_a^{m_{a,i}}
\end{equation}
with integers $m_{a,i}$ and $e_i$.

\subsubsection{Case 1: multiple cusps} \label{SSS:case 1 hauptmodul}

First assume that $\Gamma$ has at least two cusps.   We will use the following lemma to construct a hauptmodul for certain genus $0$ congruence subgroups.

Let $\OO_1, \cdots, \OO_n$ be the distinct $\Gamma$-orbits of $\calA_N$.
For each $\OO_i$, define the divisor $D_i:= \operatorname{div}(g_{\OO_i}^{12N})$  on $X_\Gamma$.
By Lemma~\ref{L:explicit D}, the divisors $D_1,\ldots,D_n$ are supported on $\{P_1,\ldots, P_r\}$ and are straightforward to compute.

\begin{lemma} \label{L:explicit hauptmodul}
Suppose there is an $n$-tuple $m \in \ZZ^n$ such that 
\[
\sum_{i=1}^n m_i  D_i = -12N\cdot P_1 + 12N\cdot P_2
\]  
Let $0\leq e < 2N^2$ be the integer satisfying  $e\equiv \sum_{i=1}^n m_i \sum_{a\in \OO_i} Na_2(N-Na_1) \pmod{2N^2}.$   Then
\[
h := \zeta_{2N^2}^e \prod_{i=1}^n g_{\OO_i}^{m_i}
\]
is a hauptmodul for $\Gamma$ whose $q$-expansion has coefficients in $K_N$. On $X_\Gamma$, we have $\operatorname{div}(h)=-P_1+P_2$. 
\end{lemma}
\begin{proof}
Since $X_\Gamma$ has genus $0$, there is a meromorphic function $f$ on $X_\Gamma$ with $\operatorname{div}(f)=-P_1+P_2$.
Lemma~\ref{L:power of g} implies that $f^{12N}/h^{12N}$ defines a function on $X_\Gamma$; it has divisor 
\[
12N \operatorname{div}(f) - \sum_{i=1}^n m_i \operatorname{div}(g_{\OO_i}^{12N}) = 12N (-P_1+P_2) - \sum_{i=1}^n m_i D_i =0,
\]
where the last equality uses our assumption on $m$.
Therefore, $f^{12N}/h^{12N}$ is constant.
Since $f$ and~$h$ are meromorphic functions on the upper half-plane, we deduce that $f/h$ is a (non-zero) constant.
In particular, $h$ is modular for $\Gamma$ and $\operatorname{div}(h)=-P_1+P_2$.
The function $h$ on $X_\Gamma$ is a hauptmodul for $\Gamma$ since its only pole is the simple pole at $P_1$, i.e., the cusp at $\infty$.

It remains to show that the coefficients of $h$ lie in $K_N$.
Take any $a\in \calA_N$.
From the series defining $g_a$, we find that it is equal to the root of unity $e(a_2(a_1-1)/2) = \zeta_{2N^2}^{Na_2(Na_1-N)}$ times a Laurent series in $q^{1/(6N^2)}$ with coefficients in $K_N$.   Set $e':=\sum_{i=1}^n m_i \sum_{a\in \OO_i} Na_2(Na_1-N)$.
The coefficients of $\zeta_{2N^2}^{-e'}\prod_{i=1}^n g_{\OO_i}^{m_i}$ thus all lie in $K_N$.
The lemma follows since $e \equiv -e' \pmod{2N^2}$.
\end{proof}

Using the Cummins-Pauli classification of genus $0$ congruence subgroups \cite{MR2016709}, we have explicitly verified that the $n$-tuple $m$ from Lemma~\ref{L:explicit hauptmodul} always exists.
Using Lemma~\ref{L:explicit D}, the existence of $m$ comes down to finding integral solutions to $r$ linear equations with integer coefficients in $n$ variables.
Using Lemma~\ref{L:explicit hauptmodul}, we can thus find an explicit hauptmodul for $\Gamma$ of the form (\ref{E:hauptmodul sum}) with $M=1$ (we have $m_{a,i}=m_i$ if $a\in \OO_i$).

\begin{remark}
One can also abstractly prove the existence of the $n$-tuple $m$.
If $N$ is an \emph{odd} prime power, then any modular function of level $N$ whose zeros and poles are all cusps can be expressed as a constant times a product of Siegel functions $g_a$ with $a\in N^{-1} \ZZ^2-\ZZ$;  cf.~Theorem~1.1(i) of Chapter 5 of \cite{MR648603}. 

If $N\geq 4$ is a power of $2$, this can also be deduced from Theorem~1.1(i) of Chapter 5 of \cite{MR648603}.
(One needs to be a little careful here since $g_a$ has a different definition in \cite{MR648603}*{Ch.~4~\S1} when $2a \in \ZZ$.
For the alternate $g_a$ from loc.~cit.~with $2a\in \ZZ$, one can express them as a constant times a product of Siegel functions $g_{a'}$ with $a' \in \calA_4 \subseteq \calA_N$.)   

The case $N=2$ can be handled directly.   For example, one can show that 
\[
g_{(1/2,0)}^8 \cdot g_{(1/2,1/2)}^4 \quad \text{ and }\quad g_{(1/2,0)}^{12} \cdot  g_{(1/2,1/2)}^{12}
\]
are hauptmoduls for $\Gamma(2)$ and $\Gamma_0(2)$, respectively (note that $\Gamma_{\rm ns}(2)$ has a single cusp and does not fall into this case, it falls into case 2 below).
\end{remark}

\subsubsection{Case 2: a single cusp and $N\neq 11$}

Now assume that $X_\Gamma$ has a single cusp and that $N\neq 11$.
There are no non-constant modular functions for $\Gamma$ whose zeros and poles are only at the cusps of $X_\Gamma$.
In particular, a hauptmodul of $\Gamma$ will never be equal to a product of Siegel functions.

Using the Cummins-Pauli classification, we find that there is a congruence subgroup $\Gamma'$ that is a proper normal subgroup of $\Gamma$, also of level $N$ and containing $-I$, such that $X_{\Gamma'}$ has genus $0$ and has exactly $[\Gamma:\Gamma']$ cusps.
(This is where we use $N\neq 11$.)

Since $X_{\Gamma'}$ has multiple cusps, we know from \S\ref{SSS:case 1 hauptmodul} how to construct a hauptmodul $h'$ of $\Gamma'$ with coefficients in $K_N$ that is of the form (\ref{E:hauptmodul sum}).
Using that $\Gamma'$ is normal in $\Gamma$, we find that $h'|_\gamma$ is modular for $\Gamma'$ for all $\gamma \in \Gamma$ and the function depends only on the coset $\Gamma'\cdot \gamma$.
Define
\[
h := \sum_{\gamma\in \Gamma'\backslash \Gamma} h'|_\gamma;
\]
it is a modular function for $\Gamma$.  Since $X_\Gamma$ has only one cusp and $X_{\Gamma'}$ has $[\Gamma:\Gamma']$ cusps, we deduce that the modular functions $\{h'|_{\gamma}\}_{\gamma\in \Gamma'\backslash \Gamma}$ on $X_{\Gamma'}$ each have their unique (simple) pole at different cusps.
This implies that $h$ has a simple pole at the unique cusp of $X_\Gamma$ and is holomorphic elsewhere.
Therefore, $h$ is a hauptmodul for $\Gamma$.   

Since $h'$ is modular for $\Gamma(N)$ and has coefficients in $K_N$, so does $h'|_\gamma$ for all $\gamma\in \SL_2(\ZZ)$; see Proposition~\ref{P:FN Galois}. 
Therefore, the coefficients of $h$ lie in $K_N$.\medskip

Finally, it remains to show that $h$ is of the form (\ref{E:hauptmodul sum}).  It suffices to show that $h'|_\gamma$ is of the form (\ref{E:hauptmodul sum}) for a fixed $\gamma\in \Gamma$.  
We know that $h'$ is equal to some product $\zeta_{2N^2}^{e} \prod_{a\in \calA_N} g_a^{m_{a}}$, so $h|_\gamma = \varepsilon(\gamma)^b \zeta_{2N^2}^e \prod_{a\in \calA_N} g_{a\gamma}^{m_a}$ with $b:= \sum_{a\in \calA_N} m_a$ by Lemma~\ref{L:Siegel basics}(\ref{L:Siegel basics iii}).
Recall that for each $a\in \calA_N$, there is a unique $a*\gamma \in \calA_N$ such that $a\gamma$ lies in the same coset of $(N^{-1}\ZZ^2)/\ZZ^2$ as $a*\gamma$ or $-a*\gamma$.
From Lemma~\ref{L:Siegel basics}(\ref{L:Siegel basics i}) and (\ref{L:Siegel basics ii}), the functions $g_{a\gamma}^{m_a}$ and $g_{a*\gamma}^{m_a}$ agree up to a multiplication by some computable root of unity $-\zeta_{N}^{e'}$.
Therefore, $h|_\gamma$ is equal to $\varepsilon(\gamma)^b$ times a function of the form (\ref{E:hauptmodul sum}) with $M=1$.   

It remains only to show that $\varepsilon(\gamma)^b$ is a power of a  $2N^2$-th root of unity.    In \cite{MR648603}*{Ch.~3~\S5}, Kubert and Lang give a necessary and sufficient condition for the product $\prod_{a\in \calA_N} g_{a}^{m_a}$ to be modular for $\Gamma(N)$; these conditions hold since $h'$ is modular for $\Gamma'\supseteq \Gamma(N)$.
If $N$ is a power of a prime $\ell\geq 5$, then Theorem~5.2 of Chapter~3 of \cite{MR648603} implies that $b \equiv 0 \pmod{12}$ and hence $\varepsilon(\gamma)^b=1$.
If $N$ is a power of $3$, then Theorem~5.3 of Chapter~3 of \cite{MR648603} implies that $b \equiv 0 \pmod{4}$ and hence $\varepsilon(\gamma)^b$ is a power of $\zeta_3$.
If $N$ is a power of $2$, then Theorem 5.3 of Chapter~3 of \cite{MR648603} implies that $b\equiv 0 \pmod{3}$ and hence $\varepsilon(\gamma)^b$ is a power of $\zeta_4$.
Therefore, $\varepsilon(\gamma)^b$ is indeed a power of a $2N^2$-th root of unity.

\subsubsection{Case 3: $N=11$}

The remaining case is when $X_\Gamma$ has a single cusp and $N=11$.  We include this case only for completeness; we will not need it for our application.

Define the function
\[
f(\tau):= \prod_{(a_1,a_2)\in B} g_{(a_1/11,\,a_2/11)}(\tau),
\]
where
\[
B:=\left\{\begin{array}{c} 
 ( 0, 1 ), ( 0, 2 ), ( 0, 3 ), ( 1, 0 ), ( 1, 2 ), ( 1, 5 ), ( 1, 7 ), ( 2, 1 ), ( 2, 2 ), ( 2, 4 ), ( 2, 5 ), ( 2, 6 ),\\ ( 2, 7 ), ( 2, 8 ), ( 2, 9 ), ( 2, 10 ), ( 3, 0 ), ( 3, 2 ), ( 3, 4 ), ( 3, 5 ), ( 3, 6 ), ( 3, 8 ), ( 3, 10 ), ( 4, 0 ), \\( 4, 1 ), ( 4, 2 ), ( 4, 4 ), ( 4, 5 ), ( 4, 6 ), ( 5, 1 ), ( 5, 4 ), ( 5, 5 ), ( 5, 6 ), ( 5, 7 ), ( 5, 8 ), ( 5, 9 ) 
\end{array}\right\}.
\]
One can verify that 
\[
\sum_{(a_1,a_2)\in B} a_1^2 \equiv \sum_{(a_1,a_2)\in B} a_2^2 \equiv \sum_{(a_1,a_2)\in B} a_1 a_2 \equiv 0 \pmod{11}
\]
and that $|B|=36\equiv 0 \pmod{12}$.   Theorem~5.2 of \cite{MR648603}*{Ch.~3 \S5} implies that $f$ is a modular function for $\Gamma(11)$.   Using $\sum_{(a_1,a_2)\in B} a_2/11\cdot (a_1/11-1)/2=-60/11$ and the $q$-expansion of Siegel functions from \S\ref{SS:Siegel}, we find that all the coefficients of $f$ lie in $K_{11}$.  Therefore, $f\in \calF_{11}$.

Using that $\Gamma(11)$ is normal in $\Gamma$, we find that $f|_\gamma$ is modular for $\Gamma(11)$ for all $\gamma \in \Gamma$ and the function depends only on the coset $\Gamma(11)\cdot \gamma$.
Define
\[
h := \sum_{\gamma\in \Gamma(11)\backslash \Gamma} f|_\gamma;
\]
it is a modular function for $\Gamma$.  That $h$ is of the form (\ref{E:hauptmodul sum}) follows as in the previous case.

We claim that $h$ is a hauptmodul for $\Gamma$.  From our description of $h$ in terms of Siegel functions, we find that $h$ has no poles except possibly at the unique cusp (at $\infty$).   From Cummins and Pauli \cite{MR2016709}, there is a unique genus $0$ congruence subgroup of $\SL_2(\ZZ)$ of level $11$ up to conjugacy in $\GL_2(\ZZ)$ (the one labeled $\text{11A}^0$).  We have computed all the possible $\Gamma$ and shown that $h$ has a simple pole at $\infty$, and is therefore a hauptmodul.

\begin{remark}
The set $B$ comes from \S5.3 of \cite{MR2059637}.   In \cite{MR2059637}, methods are given to compute hauptmoduls for genus $0$ congruence subgroups (unfortunately, their hauptmodul tables are no longer available).   They use ``generalized Dedekind eta functions'' which are essentially Siegel functions.
\end{remark}


\subsection{Rational function \texorpdfstring{$J(t)$}{\textit{J}(\textit{t})}} \label{SS:hauptmoduls J}

For a hauptmodul $h$ of $\Gamma$, there is a unique function $J(t)\in \CC(t)$ such that $J(h)=j$; it has degree $d:=[\SL_2(\ZZ):\pm \Gamma]$.    

Let us briefly explain how to compute $J(t)$ assuming that one can compute sufficiently many terms of the expansion of $f$.   Let $K\subseteq \CC$ be a field containing all the coefficients of $h$.  
Consider the equation
\begin{equation} \label{E:J equation}
(a_d h^d + \ldots + a_1 h + a_0) - j \cdot (b_d h^d+\cdots +b_1 h +b_0) = 0
\end{equation}
with unknowns $a_i, b_i \in K$, where $d:=[\SL_2(\ZZ):\pm \Gamma]$.
Computing the $q$-expansion coefficients of the left-hand side of (\ref{E:J equation}) yields a system of homogeneous linear equations in the unknowns~$a_i$ and~$b_i$.
The existence and uniqueness of $J$ ensure that the solutions $(a_1,\ldots, a_d, b_1,\ldots, b_d) \in K^{2d}$ form a one-dimensional subspace.
By computing sufficiently many coefficients of (\ref{E:J equation}) one can find a non-zero solution $(a_1,\ldots, a_d, b_1,\ldots, b_d) \in K^{2d}$, unique up to scaling by $K^\times$, and
\[
J(t)= (a_d t^d + \ldots + a_1 t + a_0)/(b_d t^d+\cdots +b_1 t +b_0) \in K(t)
\]
is then the unique rational function for which $J(h)=j$.  Note that if the hauptmodul $h$ is constructed as in the previous section then we will have $J(t)\in K_N(t)$, where $N$ is the level of $\Gamma$.

\section{Modular curves of genus 0}  \label{S:genus 0 check}

Fix the following:
\begin{itemize}
\item An integer $N>1$ that is a prime power.
\item A subgroup $G$ of $\GL_2(\ZZ/N\ZZ)$ satisfying $-I \in G$ and $\det(G)=(\ZZ/N\ZZ)^\times$.    
\item A rational function $J(t)\in \QQ(t)$.
\end{itemize}
In this section, we explain how to determine if the function field of $X_G$ is of the form $\QQ(f)$ for some modular function $f\in \calF_N$ satisfying $J(f)=j$.   We will use this to verify the entries of Tables \ref{table:g0three}--\ref{table:g0two}.
\medskip

If such an $f$ exists, then $X_G$ is isomorphic to $\PP^1_\QQ$ and the morphism $\pi_G\colon X_G\to \PP^1_\QQ$ is given by the relation $j=J(f)$ in their function fields.  We may assume the necessary condition that $[\GL_2(\ZZ/N\ZZ):G]=\deg \pi_G$ agrees with the degree of $J(t)$.   

\begin{remark}
In \S\ref{S:how} we explain how the $J(t)$ listed in Tables \ref{table:g0three}--\ref{table:g0two}, were actually found, which involves the use of a Monte Carlo algorithm and assumes the Generalized Riemann Hypothesis (GRH).  The purpose of this section is to explain how we can unconditionally verify a given $J(t)$, regardless of how it was found.
\end{remark}

\subsection{Construction of possible \texorpdfstring{$f$}{\textit{f}}}

Let $\Gamma$ be the congruence subgroup consisting of $\gamma\in \SL_2(\ZZ)$ for which $\gamma^t$ modulo $N$ lies in $G$ (equivalently, in $G \cap \SL_2(\ZZ/N\ZZ)$).   By Lemma~\ref{L:modular connection}(\ref{L:modular connection ii}), we may assume that $\Gamma$ has genus $0$  since otherwise  $X_G$ has positive genus and its function field cannot be of the form $\QQ(f)$.

The group $\Gamma$ acts on the right on the field $\calF_N$; let $\calF_N^{\Gamma}$ be subfield  fixed by this action.   By Lemma~\ref{L:modular connection}(\ref{L:modular connection i}), we have $K_N(X_G)=\calF_N^\Gamma$.

In \S\ref{SS:hauptmodul construction}, we described how to compute an explicit hauptmodul $h$ for $\Gamma$ such that  coefficients of its $q$-expansion all lie in $K_{N'}\subseteq K_N$, where the level $N'$ of $\Gamma$ divides $N$.   Therefore, we have
\[
K_N(X_G)=\calF_N^\Gamma = K_N(h).
\]

Moreover, we can express $h$ in terms of Siegel functions and hence we can compute as many of its coefficients as we desire.    In \S\ref{SS:hauptmoduls J}, we described how to compute the unique rational function $J'(t)\in K_N(t)$ for which $j=J'(h)$.  The degree of $J'(t)$ agrees with $[\SL_2(\ZZ):\Gamma]=[\GL_2(\ZZ/N\ZZ):G]$, thus $J(t)$ and $J'(t)$ have the same degree.

\begin{remark}
The rational function $J'(t)$ gives a map to the $j$-line from $X_\Gamma$, which is defined over $K_N=\QQ(\zeta_N)$, while the rational function $J(t)$ gives a map to the $j$-line from $X_G$, which is defined over $\QQ$.
We use $J'(t)$ in our procedure to verify $J(t)$, but note that $J'(t)$ does not determine $J(t)$; in general there will be multiple non-conjugate subgroups $G$ corresponding to $\Gamma$ and a different rational function $J(t)$ for each of the corresponding $X_G$ (in total we have 220 modular curves $X_G$ of genus 0 corresponding to 73 modular curves $X_\Gamma$).
\end{remark}

\begin{lemma} \label{L:psi intro}
The modular functions $f\in K_N(X_G)$ that satisfy $K_N(X_G)=K_N(f)$ and $J(f)=j$ are precisely those of the form $\psi(h)$, where $\psi(t)\in K_N(t)$ is a degree $1$ function satisfying $J'(t)=J(\psi(t))$.
\end{lemma}
\begin{proof}
First take any $\psi(t)\in K_N(t)$ of degree $1$ satisfying $J'(t)=J(\psi(t))$.   Define $f:=\psi(h)$.  We have $K_N(f)=K_N(h)=K_N(X_G)$, since $\psi$ has an inverse, and $J(f)=J(\psi(h))=J'(h)=j$.

Now suppose that $K_N(X_G)=K_N(f)$ for some $f\in K_N(X_G)$ satisfying $J(f)=j$.   Since $K_N(f)=K_N(X_G)=K_N(h)$, we have $f=\psi(h)$ for a unique $\psi(t)\in K_N(t)$ of degree $1$.   We then have $j=J(f)=J(\psi(h))$ and therefore $J'(t)=J(\psi(t))$, since $J'(t)$ is the unique element of $K_N(t)$ that satisfies $J'(h)=j$.
\end{proof}

\subsection{Finding possible \texorpdfstring{$f$}{\textit{f}}} \label{SS:finding}

Define $\Psi$ to be the set of $\psi(t)\in K_N(t)$ of degree $1$  for which $J'(t)=J(\psi(t))$; these $\psi$ arise in Lemma~\ref{L:psi intro}.   We now explain how to compute $\Psi$.

Choose three distinct elements $\beta_1,\beta_2, \beta_3 \in K_N \cup \{\infty\}$.    For $1\leq i \leq 3$, define the set 
\[
R_i :=\{ \alpha \in K_N \cup \{\infty\} : J'(\beta_i)=J(\alpha) \text{ and } \ord_{\beta_i}(J') = \ord_{\alpha} (J) \},
\]
where $\ord_{\beta_i}(J')$ is the order of vanishing of $J'(t)$ at $t=\beta_i$.  
Let $R$ be the set of triples $\alpha=(\alpha_1,\alpha_2,\alpha_3)\in R_1\times R_2\times R_3$ such that $\alpha_1$, $\alpha_2$ and $\alpha_3$ are distinct.
Let $\psi_\alpha \in K_N(t)$ be the \emph{unique} rational function of degree $1$ such that $\psi_\alpha(\beta_i)=\alpha_i$ for all $1\leq i\leq 3$.

Take any $\psi\in \Psi$.  We have $J'(\beta_i)=J(\psi(\beta_i))$ and $\ord_\beta(J')=\ord_{\psi(\beta)}(J)$ for each $1\leq i \leq 3$.   Therefore, $\psi(\beta_i) \in R_i$ for each $1\leq i \leq 3$ and hence $\psi = \psi_\alpha$ for some $\alpha \in R$.    So we have
\[
\Psi = \{ \psi_\alpha : \alpha \in R, \, J'(t)=J(\psi(t))\}.
\]
Since $R$ is finite, this gives us a way to compute the (finite) set $\Psi$.

By Lemma~\ref{L:psi intro}, the set
\[
\{ \psi(h): \psi \in \Psi\}
\]
is the set of modular functions $f\in K_N(X_G)$ that satisfy $K_N(X_G)=K_N(f)$ and $J(f)=j$.

\subsection{Checking each \texorpdfstring{$f$}{\textit{f}}} \label{SS:checking f}

Let $f$ be one of the finite number of functions that satisfy $K_N(X_G)=K_N(f)$ and $J(f)=j$.   We just saw how to compute all such $f$; they are of the form $\psi(h)$ for a degree $1$ function $\psi(t)\in K_N(t)$ and a modular function $h$ satisfying $K_N(X_G)=K_N(h)$ that is expressed in terms of Siegel functions.
Recall from \S\ref{S:modular functions} that each $A\in \GL_2(\ZZ/N\ZZ)$ acts on $\calF_N$ via the isomorphism $\theta_N \colon \GL_2(\ZZ/N\ZZ)/\{\pm I\} \xrightarrow{\sim} \Gal(\calF_N/\QQ(j))$ of Proposition~\ref{P:FN Galois}, and for $f\in \calF_N$ we use $A_*(f):=\theta_N(A)(f)$ to denote this action.
\bigskip\bigskip 

\begin{lemma} \label{L:checking f}
\begin{romanenum}
\item \label{L:checking f i}
We have $\QQ(X_G)=\QQ(f)$ if and only if $f \in \QQ(X_G)$.
\item \label{L:checking f ii}
For a matrix $A\in G$, we have $A_*(f)=f$ if and only if $\ord_q(A_*(f)-f)> 2w/N'$, where $w$ is the width of the cusp $\infty$ of $X_\Gamma$ and $N'$ is the level of $\Gamma$.
\end{romanenum}
\end{lemma}
\begin{proof}
We first prove part (\ref{L:checking f i}); only one implication needs proof.  Suppose that $f\in \QQ(X_G)$.    Then $\QQ(f) \subseteq \QQ(X_G)$ and it suffices to show that these two fields have the same degree over $\QQ(j)$.   This is true since we have been assuming that $\deg \pi_G$ is equal to the degree of $J(t)$.

We now prove part (\ref{L:checking f ii}); only one implication needs proof.  Suppose that $\ord_q(A_*(f)-f)> 2w/N'$.  As meromorphic functions on $X_\Gamma$, $f$ and $A_*(f)$ have a unique (simple) pole since $h$ has this property and $\psi$ has degree $1$.  Therefore, the function $A_*(f)-f$ on $X_\Gamma$ is zero or has at most two poles (and hence at most two zeros).   Our assumption $\ord_q(A_*(f)-f)> 2w/N'$ implies that $A_*(f)-f$ has a zero of order $3$ at the cusp $\infty$ and thus $A_*(f)-f=0$.
\end{proof}

By Lemma~\ref{L:checking f}(\ref{L:checking f i}), we have $\QQ(X_G)=\QQ(f)$ if and only if $A_*(f)=f$ for all $A\in G$ in a set of generators of $G$; it suffices to consider $A\in G$ for which $\det(A)$ generate $(\ZZ/N\ZZ)^\times$ since $h$ and hence $f$ is fixed by $G\cap\SL_2(\ZZ/N\ZZ)$.  It remains to describe how to determine whether $A_*(f)$ is equal to $f$.   By Lemma~\ref{L:checking f}(\ref{L:checking f ii}), it suffices to compute enough terms of the $q$-expansion of $A_*(f)-f$ to determine whether $\ord_q(A_*(f)-f)> 2w/N'$ holds.

Finally, let us briefly explain how to compute terms in the $q$-expansion of $A_*(f)-f$.
Let $d$ be an odd integer congruent to $\det(A)$ modulo $N$.
Choose a matrix $\gamma\in \SL_2(\ZZ)$ so that  $A^t\equiv \big(\begin{smallmatrix} 1 & 0 \\ 0 & d\end{smallmatrix}\big) \, \gamma \pmod{N}$.  We thus have
\begin{equation} \label{E:gf-f expansion}
A_*(f) -f = \sigma_d(f)|_\gamma -f  = \sigma_{d}(\psi)(\sigma_d(h)|_\gamma ) - \psi(h), 
\end{equation}
where $\sigma_d(\psi)$ is the rational function with $\sigma_d$ applied to the coefficients of its numerator and denominator.
Our hauptmodul $h$ is of the form $\sum_{i=1}^M \zeta_{2N^2}^{e_i} \prod_{a\in \calA_N} g_a^{m_{a,i}}$ for certain integers $e_i$ and $m_{a,i}$, so 
\[
\sigma_d(h)|_\gamma= \sum_{i=1}^M \zeta_{2N^2}^{e_i d} \prod_{a\in \calA_N} \big(\sigma_{d}(g_a)|_\gamma\big)^{m_{a,i}}.
\]
From the series expansion of $g_a$, one easily checks that $\sigma_d(g_{(a_1,a_2)})= g_{(a_1,da_2)}$.
From Lemma~\ref{L:Siegel basics}(\ref{L:Siegel basics iii}), we have $\sigma_{d}(g_a)|_\gamma = \varepsilon(\gamma) g_{(a_1,da_2)\gamma}$ and hence
\[
\sigma_d(h)|_\gamma= \sum_{i=1}^M \zeta_{2N^2}^{e_i d} \cdot \prod_{a\in \calA_N}\varepsilon(\gamma)^{m_{a,i}} \cdot \prod_{a\in \calA_N} g_{(a_1,da_2)\gamma}^{m_{a,i}}.
\]
Thus by computing enough terms in the $q$-expansion of the functions $\{g_a\}_{a\in \calA_N}$, we are able to compute the $q$-expansion of $h$ and $\sigma_d(h)|_\gamma$ to as many terms as we desire.
This allows us to compute terms in the $q$-expansion of $A_*(f)-f$ via (\ref{E:gf-f expansion}).  

\begin{remark}
Suppose that $X_\Gamma$ has at least $3$ cusps.  We then have $A_*(f)=f$ if and only if $A_*(f)$ and $f$ take the same the value at any three of the cusps (as in the proof of Lemma~\ref{L:checking f} this implies that $A_*(f)-f$ has at least three zeros and hence is the zero function).    In the case of at least $3$ cusps, our hauptmodul $h$ was given as a constant times a product of Siegel functions; so its value at the cusp $\infty$ is determined by the first term of the $q$-expansion of $h$.   The value at any other cusp~$c$ can be determined by the first term of the $q$-expansion of $h|_{\gamma}$ with $\gamma\in \SL_2(\ZZ)$ satisfying $\gamma\infty=c$.    This approach is quicker since fewer terms of the $q$-expansions are required.
\end{remark}

\subsection{Verifying the entries of our tables}\label{SS:verify}

We now explain how to verify the validity of our genus~$0$ tables.  \texttt{Magma} scripts that perform these verifications can be found at \cite{MagmaScripts}.

Each row of Tables \ref{table:g0three}--\ref{table:g0two} gives a set of generators of a subgroup $G$ of $\GL_2(\ZZ/N\ZZ)$ that satisfies $-I\in G$ and $\det(G)=(\ZZ/N\ZZ)^\times$ for a prime power $N$.  We may assume that $N>1$.   By composing rational maps,  we obtain a corresponding rational function $J(t)\in \QQ(t)$.

Using the earlier parts of \S\ref{S:genus 0 check}, we can construct a modular function $f\in \calF_N$ such that $\QQ(X_G)=\QQ(f)$ and $J(f)=j$.  So  $X_G$ is isomorphic to $\PP^1_\QQ$ and the morphism $\pi_G\colon X_G\to \PP^1_\QQ$ is given by the relation $j=J(f)$ in their function fields.  (We also note that there is no harm in replacing $G$ by a conjugate group; this is useful because one can reuse the hauptmodul computations for different groups in the tables).\medskip

Fix a group $G \subseteq \GL_2(\ZZ/N\ZZ)$ as above, and a modular function $f\in\calF_N$ satisfying $\QQ(X_G)=\QQ(f)$ and $J(f)=j$.    

Now fix another group $G' \subseteq \GL_2(\ZZ/N'\ZZ)$ from our table so that $N$ divides $N'$ and the image of $G'$ in $\GL_2(\ZZ/N\ZZ)$ is conjugate to a subgroup of $G$.   In the above computations, we have constructed a modular function $f'$ satisfying $\QQ(X_{G'})=\QQ(f')$ and $J'(f')=j$ for a rational function $J'(t) \in \QQ(t)$ also arising from the tables.  

Take any subgroup $\widetilde G\subseteq \GL_2(\ZZ/N\ZZ)$  conjugate to $G'$ whose image modulo $N$ lies in $G$.   Choose any $A\in \GL_2(\ZZ/N'\ZZ)$ for which $\tilde G:=AG'A^{-1}$ and define $\tilde f:=A_*(f')$.  We have an inclusion of fields  
\[
\QQ(\tilde f)=\QQ(X_{\widetilde G})\supseteq \QQ(X_G)=\QQ(f).
\] 
The extension $\QQ(\tilde f)/\QQ(f)$ has degree $i:= [\GL_2(\ZZ/N'\ZZ)\!:\!G']/[\GL_2(\ZZ/N\ZZ)\!:\!G]$. Therefore, $\varphi(\widetilde{f})=f$ for a unique $\varphi(t) \in \QQ(t)$ of degree $i$.   We can compute $\varphi(t)$ using the method from \S\ref{SS:hauptmoduls J}; the coefficients of $f$ and $\widetilde f$ can be computed as in \S\ref{SS:checking f}.

The rational function $\varphi$ is not unique, it depends on the choices of $\widetilde G$, $f$, $f'$, and $A$.   However, any other rational function occurring would be of the form $\psi'(\varphi(\psi(t)))$, where $\psi,\psi'\in \QQ(t)$ are degree~$1$ functions satisfying $J(\psi(t))=J(t)$ and $J'(\psi'(t))=J'(t)$.   Note that all the possible $\psi$ and $\psi'$ can be computed as in \S\ref{SS:finding} (with $J=J'$).     We have checked that the rational function relating $G$ and $G'$ in our tables, when given, is indeed of the form $\psi'(\varphi(\psi(t)))$.

\section{Modular curves of genus 1}
\label{S:genus1}

We now consider the open subgroups $G$ of $\GL_2(\Zhat)$ with genus $1$ and prime power level $N=\ell^e$ that satisfy $-I\in G$ and $\det(G)=\Zhat^\times$.    We are interested in describing those $G$ for which $X_G(\QQ)$ is infinite.    There is no harm in replacing $G$ by a conjugate.   So by Theorem~\ref{T:groups 0 and 1}(\ref{T:groups 0 and 1 ii}), there are $250$ cases that need to be checked.
\\

Let $J_G$ be the Jacobian of the curve $X_G$.  Using the methods of \cite{PossibleIndices}, we can compute the rank of $J_G(\QQ)$.   From \cite{MR0337993}*{\S IV}, we find that the curve $X_G$ has good reduction at all primes $p\nmid N=\ell^e$.   Therefore, $J_G$ is an elliptic curve defined over $\QQ$ whose conductor is a power of $\ell$.    The primes $\ell$ that arise are small enough to ensure that $J_G$ is isomorphic to one of the elliptic curves in Cremona's tables \cite{Cremona}; this gives a finite number of candidates for $J_G$ up to isogeny.

For each prime $p\nmid 6\ell$, we can compute $\#J_G(\FF_p)=\#X_G(\FF_p)$ from the modular interpretation of $X_G$, cf.~\cite{PossibleIndices}*{\S3.6} for details.   In particular, we can compute $\#J_G(\FF_p)$ directly from the group~$G$ without computing a model for $X_G$ (or its reduction modulo $p$).  By computing several values of $\#J_G(\FF_p)$ with $p\neq \ell$, we can quickly distinguish the isogeny class of $J_G$ among the finite set of candidates.   We then compute the rank of $J_G(\QQ)$, which we note is an isogeny invariant.

Running this procedure on each of the $250$ genus $1$ groups $G$ given by Theorem~\ref{T:groups 0 and 1}, we find that $J_G(\QQ)$ has rank $0$ for $222$ groups and $J_G(\QQ)$ has positive rank for $28$ groups; a \texttt{Magma} script that performs this computation can be found in \cite{MagmaScripts}.
We need only consider the $28$ groups $G$ for which $J_G(\QQ)$ has positive rank, since $X_G(\QQ)$ is finite if $J_G(\QQ)$ has rank $0$.\medskip

Now let $G$ be one of the $28$ groups for which $J_G(\QQ)$ has positive rank; they are precisely the $28$ genus $1$ groups in Theorem~\ref{T:main} and can be found in Table~\ref{table:g1} in Appendix~\ref{S:appendix}.   For each of these groups $G$, if $X_G(\QQ)$ is nonempty then it must be infinite, since the Abel-Jacobi map then gives a bijection from $X_G(\QQ)$ to $J_G(\QQ)$.  We initially verified that $X_G(\QQ)$ is nonempty by finding an elliptic curve $E/\QQ$ with $\rho_E(\Gal_\QQ)\subseteq G$ using an extension of the algorithm in \cite{ComputingImages}.

For each of these $28$ groups $G$, a model for $X_G$ and the morphism $\pi_G$ can already be found in the literature (and are equivalent to the ones we give in Appendix~\ref{S:appendix}).  For the $27$ groups $G$ of level $16$ these curves and morphisms were constructed by Rouse and Zureick-Brown in \cite{RZB}; the models and morphisms we give in Table~\ref{table:g1} for these groups are slightly different (we constructed them by taking fiber products of our genus 0 curves), but we have verified that they are isomorphic (note that their groups are transposed relative to ours).  The remaining group $G$ has level $11$ and its image in $\GL_2(\ZZ/11\ZZ)$ is the normalizer of a non-split Cartan subgroup.  An explicit model for $X_G=X^+_{\rm ns}(11)$ and the morphism to the $j$-line can be found in \cite{MR1677158}; these are reproduced in Appendix~\ref{S:appendix}.

\section{Proof of Theorem~\ref{T:ell-adic}} \label{S:ell-adic proof}

If $\ell \leq 13$, then the set $\calJ_\ell$ is finite by \cite{PossibleIndices}*{Proposition 4.8}.  If $\ell >13$, this follows from \cite{PossibleIndices}*{Proposition 4.9}; note that $\rho_{E,\ell^\infty}$ is surjective if and only if $\rho_{E,\ell}$ is surjective, since $\ell\ge 5$, by~\cite{MR1043865}*{IV Lemma 3}.  This proves (\ref{T:ell-adic 0}).\medskip

For a group $G$ from Theorem~\ref{T:main}, define the set 
\[
\calS_G:= {\bigcup}_{G'} \pi_{G',G}( X_{G'}(\QQ)),
\]
where $G'$ varies over the proper subgroups of $G$ that are conjugate to one of groups in Theorem~\ref{T:main} of $\ell$-power level and $\pi_{G',G} \colon X_{G'}\to X_G$ is the natural morphism induced by the inclusion $G' \subseteq G$.  Note that this is a finite union.

Suppose first that $G$ has genus $0$.  Then $X_G\simeq \PP^1_\QQ$ and $\calS_G$ is a \emph{thin} subset of $X_G(\QQ)$, in the language of \cite{MR1757192}*{\S9}.
The field $\QQ$ is Hilbertian, and in particular $\PP_1(\QQ)\simeq X_G(\QQ)$ is not thin; this implies that the complement $X_G(\QQ)-\calS_G$ cannot be thin and must be infinite.

Suppose that $G$ has genus $1$.   If $G$ does not have level $16$ and index $24$, then there are no proper subgroups $G'$ of $G$ that are conjugate to a group from Theorem~\ref{T:main}, and therefore $\calS_G$ is empty and $X_G(\QQ)-\calS_G$ is infinite.

Now suppose that $G$ has genus $1$, level $16$, and index $24$.  There are 7 such $G$, labeled
\[
\text{16C}^1\text{-16c},\ \text{16C}^1\text{-16d},\ \text{16B}^1\text{-16a},\ \text{16B}^1\text{-16c},\ \text{16D}^1\text{-16d},\ \text{8D}^1\text{-16b},\ \text{8D}^1\text{-16c}
\]
and explicitly described in Table~\ref{table:g1} of Appendix~\ref{S:appendix}.
Each of these $G$ contains either~2 or~4 index~2 subgroups $G'$ that are conjugate to one of the groups in Theorem~\ref{T:main}.
In every case we have $\calS_G=X_G(\QQ)$, so that $X_G(\QQ)-\calS_G$ is empty; see Example 6.11 and Remark 6.3 in \cite{RZB}.\medskip

Let $E/\QQ$ be an elliptic curve with $j_E \notin \calJ_\ell$.   The group $\pm \rho_{E,\ell^\infty}(\Gal_\QQ)$ is conjugate in $\GL_2(\ZZ_\ell)$ to the $\ell$-adic projection of a unique group $G$ from Theorem~\ref{T:main} with $\ell$-power level.     Using Proposition~\ref{P:XG rational}, we can also characterize $G$ as the unique group  from Theorem~\ref{T:main} with $\ell$-power level such that $j_E \in  \pi_G(X_G(\QQ) - \calS_G)$.  Parts (\ref{T:ell-adic i}) and (\ref{T:ell-adic ii}) follow by noting that $ \pi_G(X_G(\QQ) - \calS_G)$ is empty when $G$ has genus $1$, level $16$, and index $24$, and it is infinite otherwise.
 
\section{How the \texorpdfstring{${J(t)}$}{{\textit J}({\textit t})} were found}
\label{S:how}

Let $G$ be one of the genus $0$ subgroups of $\GL_2(\Zhat)$ from Theorem~\ref{T:main};  they are listed in Tables~\ref{table:g0three}--\ref{table:g0two} in Appendix~\ref{S:appendix} and were determined using the algorithm described in \S\ref{S:group theory}.  For each $G$, we also have a rational function $J(t)\in \QQ(t)$ such that the function field of $X_G$ is of the form $\QQ(f)$ and $j=J(f)$, where $j$ is the modular $j$-invariant; the verification of this property is described in \S\ref{S:genus 0 check}.

In this section, we explain how we found $J(t)$; note that the method we used to verify the correctness of $J(t)$ does not depend on how it was found!  None of our theorems depend on the techniques described in this section.  All that matters is that they eventually produced functions $J(t)$ whose correctness we could verify using the procedure described in \S\ref{SS:verify}.\medskip

We used an extension of the algorithm in \cite{ComputingImages} to search for elliptic curves $E/\QQ$ for which $\rho_{E}(\Gal_\QQ)$ is conjugate to a subgroup of $G$.  This was initially done by simply checking elliptic curves in Cremona's tables \cite{Cremona} and the LMFDB \cite{LMFDB} (but see Remark~\ref{R:search} below).
 After enough searching, we find elliptic curves $E_1$, $E_2$, $E_3$ defined over~$\QQ$ with distinct $j$-invariants $j_1$, $j_2$, $j_3$ for which we believe that $\rho_{E_i}(\Gal_\QQ)$ is conjugate in $\GL_2(\Zhat)$ to a subgroup of $G$; in particular, we expect that $j_1,j_2,j_3 \in \pi_G(X_G(\QQ))$. 
We ran the Monte Carlo algorithm in \cite{ComputingImages} using parameters that ensure the error probability is less than $2^{-100}$, under the GRH.

Now suppose that $j_1$, $j_2$, $j_3$ are indeed elements of $\pi_G(X_G(\QQ))$. The curve $X_G$ has genus $0$ and rational points, so it is isomorphic to $\PP^1_\QQ$.   We can choose an isomorphism $X_G\simeq \PP^1_\QQ$ such that there are points $P_1, P_2, P_3 \in X_G(\QQ)$ satisfying $\pi_G(P_i)=j_i$ which map to $0$, $1$, $\infty$, respectively.  There is thus a rational function $J(t)\in \QQ(t)$ such that $J(0)=j_1$, $J(1)=j_2$, $J(\infty)=j_3$ and such that $\QQ(X_G)=\QQ(f)$ for a modular function $f$ satisfying $J(f)=j$; the function $f$ is obtained by composing our isomorphism $\PP^1_\QQ\simeq X_G$ with $\pi_G$.

We can now find all such potential $J$.  As explained in \S\ref{S:genus 0 check}, we can construct a modular function $h\in \calF_N$ and a rational function $J'(t)\in K_N(t)$ such that $K_N(X_G)=K_N(h)$ and $j=J'(h)$, where~$N$ is the level of $G$.   
We thus have
 \[
 J(t)=J'(\psi(t))
 \] 
for some degree $1$ function $\psi(t)\in K_N(t)$ satisfying $\psi(0) \in R_1$, $\psi(1)\in R_2$ and $\psi(\infty) \in R_3$, where $R_i:=\{\alpha\in K_N\cup \{\infty\}: J'(\alpha)=j_i\}$.    Since the sets $R_i$ are finite and disjoint, there are only finitely many $\psi(t) \in \QQ(t)$ of degree $1$ satisfying $\psi(0) \in R_1$, $\psi(1)\in R_2$, $\psi(\infty) \in R_3$.    For each such $\psi(t)$, we check whether $J'(\psi(t))$ lies in $\QQ(t)$.

Consider any $\psi$ as above for which $J'(\psi(t))\in \QQ(t)$.  Set $J(t):=J'(\psi(t))$ and $f:=\psi^{-1}(h)\in K_N(X_G)$.    We have $J(f)=J'(h)=j$.  The field $\QQ(f)$ is thus the function field of a modular curve $X_{G'}$, where $G'$ is an open subgroup of $\GL_2(\Zhat)$ of level $N$ satisfying $\det(G')=\Zhat^\times$ and $-I\in G'$; it consists of matrices whose reduction modulo $N$ fix $f$.    We can then check whether $G$ is equal to~$G'$.  Since $[\GL_2(\Zhat):G]= \deg \pi_G = \deg J = [\GL_2(\Zhat):G']$, it suffices to determine whether $G$ is a subgroup of $G'$; equivalently, whether $G$ fixes $f$.  A method for determining whether $f$ is fixed by~$G$ is described in \S\ref{SS:checking f}.

We will eventually find a $\psi$ for which we have $G=G'$ (provided that our initial $j$-invariants~$j_i$ are valid).   This then proves that $\QQ(X_G)=\QQ(f)$ for some $f$ satisfying $J(f)=j$, where $J(t):=J'(\psi(t))\in \QQ(t)$.  

Note this rational function $J(t)$ is not unique since $J(\varphi(t))$ would also work for any $\varphi(t)\in \QQ(t)$ of degree $1$.   Using similar reasoning, it is easy to determine if two $J_1,J_2\in \QQ(t)$ satisfy $J_2(t)=J_2(\varphi(t))$ for some degree $1$ function $\varphi\in \QQ(t)$.   We have chosen our rational functions so that they are relatively compact when written down.

\begin{remark}\label{R:search}
Having run this procedure to obtain functions $J(t)$ for each of the groups $G$ where we were able to find suitable $E_1,E_2,E_3$ in Cremona's tables, we then address the remaining groups $G$ by picking a group $G'$ that contains a subgroup conjugate to $G$ for which we already know a function $J'(t)\in \QQ(t)$; such a $G'$ existed for every $G$ not addressed in our initial search of Cremona's tables.
Using the function $J'(t)$ we can quickly obtain a large list of elliptic curves $E$ for which $\rho_{E}(\Gal_\QQ)$ is a subgroup of $G'$.
By running the algorithm in \cite{ComputingImages} on several thousand (or even millions) of these curves we are eventually able to find $E_1$, $E_2$, $E_3$ with distinct $j$-invariants for which it is highly probable that $\rho_{E_i}(\Gal_\QQ)$ is actually conjugate to a subgroup of the smaller group~$G$ contained in $G'$.  We then proceed as above to compute the function $J(t)$ for $G$.\medskip
\end{remark}

\appendix
\section{Tables}\label{S:tables}  \label{S:appendix}

This section includes tables that list data for all the groups $G$, up to conjugacy in $\GL_2(\Zhat)$, from Theorem~\ref{T:main}.  The genus $0$ groups are given in Tables~\ref{table:g0three}, \ref{table:g0odd} and \ref{table:g0two}.  The genus $1$ groups are given in Table~\ref{table:g1}.\medskip

We now describe how to read Tables~\ref{table:g0three}--\ref{table:g1}.  Each row corresponds to a unique group $G$ from Theorem~\ref{T:main} up to conjugacy; it is given a unique label of the form $MZ^g$-$Nz$, where $M$, $N$ and $g$ are integers and $Z$ and $z$ are letters.  The integers $N$ and $g$ are the level and genus of $G$, respectively.  Let $\Gamma$ be the congruence subgroup consisting of matrices in $\SL_2(\ZZ)$ whose image modulo $N$ lies in the image of $G$ modulo $N$.   The integer $g$ is also the genus of the Riemann surface obtained by taking the quotient of the complex upper-half plane by the action of $\Gamma$.   The integer $M$ is the level of $\Gamma$ and $Z$ is an uppercase letter that distinguishes $\Gamma$ up to conjugacy in $\GL_2(\ZZ)$;  the prefix $MZ^g$ for $\Gamma$ matches the label used by Cummins and Pauli \cite{MR2016709}.  The letter $z$ is chosen so that the label $MZ^g$-$Nz$ distinguishes $G$ up to conjugacy in $\GL_2(\Zhat)$.  In some  of the tables, we also number the rows.

For each row of these tables, there is a positive integer $N$ and a set of matrices $S$ in $\GL_2(\ZZ/N\ZZ)$; the corresponding open subgroup $G$ of $\GL_2(\Zhat)$ consists of the matrices $A\in \GL_2(\Zhat)$ whose image in $\GL_2(\ZZ/N\ZZ)$ lies in the subgroup generated by $S$.  The integer $i$ in each row is the index $[\GL_2(\Zhat):G]$.

For each row of a table, we also have a list of minimal supergroups up to conjugacy (given by labels or numberings of other rows); more precisely, the groups $G'$ satisfying  $G\subsetneq G' \subseteq \GL_2(\Zhat)$, up to conjugacy in $\GL_2(\Zhat)$, for which there are no subgroups strictly between $G$ and $G'$.   \\

If $G$ has genus $0$, then $X_G$ is isomorphic to $\PP^1_\QQ$; its function field is of the form $\QQ(t)$.   If $G$ has genus~$1$, then the curve $X_G$ is isomorphic to the elliptic curve in Table~\ref{table:ec} given by the ``curve'' column of Table~\ref{table:g1}; its function field is of the form $\QQ(x,y)$, where $x$ and $y$ are the given Weierstrass coordinates.\medskip

Now for simplicity, suppose that $G$ has genus $0$.  The function $F$ listed in the ``map" column describes the morphism $X_G\to X_{G'}$ for one of the corresponding minimal supergroups $G'$ of $G$ (or the first minimal supergroup listed if only one map is given).    More precisely, after appropriately conjugating $G'$ satisfying $G\subsetneq G'$, the morphism $X_G\to X_{G'}$ is given by the rational function $F(t)\in \QQ(t)$, i.e., the curves $X_G$ and $X_{G'}$ have function fields $\QQ(t)$ and $\QQ(u)$, respectively, and are related by the equation $u=F(t)$.   By composing these rational maps down to $X_{\GL_2(\Zhat)}=\PP^1_\QQ$  (i.e., to the group labeled $\text{1A}^0\text{-1a})$, we
 obtain from $G$ a rational function 
\[
J(t)\in \QQ(t);
\] 
in \S\ref{S:genus 0 check}, we showed that $J(t)$ describes the morphism $\pi_G$ from $X_G$ to the $j$-line.   If $J'(t)\in \QQ(t)$ is the rational function arising from $G'$ in the same manner, then we will have $J(t)=J'(F(t))$.
The function $J(t)$ is independent of any choice of supergroups.\medskip

Similar remarks and conventions hold when $G$ has genus $1$.
\begin{table}
\setlength{\extrarowheight}{3pt}
\fontsize{10.75pt}{10.75pt}
\begin{tabular}{r@{\hskip 10pt}l@{\hskip 6pt}r@{\hskip 6pt}r@{\hskip 8pt}l@{\hskip 6pt}l@{\hskip 4pt}r}
\# & label & $i$ & $N$ & generators & map & sup\\\toprule
0&$\text{1A}^0\text{-1a}$ & $1$ & $1$ & & &\\
1&$\text{3A}^0\text{-3a}$ & $3$ & $3$ & $\smallmat{0}{1}{2}{0},\smallmat{1}{1}{1}{2},\smallmat{1}{0}{0}{2} $ & {\small $t^3$} & 0\\
2&$\text{3B}^0\text{-3a}$ & $4$ & $3$ & $\smallmat{0}{1}{2}{1},\smallmat{1}{2}{0}{2} $ & {\small $(t+3)^3(t+27)\, /\, t$} & 0\\
3&$\text{3C}^0\text{-3a}$ & $6$ & $3$ & $\smallmat{0}{1}{2}{0},\smallmat{1}{0}{0}{2} $ & {\small $(t-9)(t+3) \,/\, t$} & 1\\
4&$\text{3D}^0\text{-3a}$ & $12$ & $3$ & $\smallmat{2}{0}{0}{2},\smallmat{1}{0}{0}{2} $ & {\small $729 \,/\, (t^3-27)$} & 2\\
&&&&& {\small$-27(t-3)\,/\,(t^2+3t+9)$} & 3\\
5&$\text{9A}^0\text{-9a}$ & $9$ & $9$ & $\smallmat{0}{2}{4}{0},\smallmat{1}{1}{4}{5},\smallmat{1}{0}{0}{2} $ & {\small $t^3+9t-6$} & 1\\
\bottomrule
\end{tabular}
\bigskip
\caption{Genus zero groups of 3-power level}\label{table:g0three}
\end{table}

\begin{table}
\setlength{\extrarowheight}{3pt}
\begin{tabular}{@{\hskip 10pt}l@{\hskip 8pt}r@{\hskip 8pt}r@{\hskip 8pt}l@{\hskip 8pt}l@{\hskip 8pt}l}
label & $i$ & $N$ & generators & map & supergroup\\\toprule
$\text{1A}^0\text{-1a}$ & $1$ & $1$ & & &\\
$\text{5A}^0\text{-5a}$ & $5$ & $5$ & $\smallmat{2}{1}{0}{3},\smallmat{1}{2}{2}{0},\smallmat{1}{1}{0}{2} $ & {\small $t^3(t^2+5t+40)$} & $\text{1A}^0\text{-1a}$\\
$\text{5B}^0\text{-5a}$ & $6$ & $5$ & $\smallmat{2}{0}{0}{3},\smallmat{1}{0}{1}{1},\smallmat{1}{0}{0}{2} $ & {\small $(t^2+10t+5)^3 \,/\, t$} & $\text{1A}^0\text{-1a}$\\
$\text{5C}^0\text{-5a}$ & $10$ & $5$ & $\smallmat{3}{1}{0}{2},\smallmat{1}{2}{2}{0},\smallmat{2}{2}{2}{1} $ & {\small $8000t^3(t+1)(t^2-5t+10)^3 \,/\, (t^2-5)^5$} & $\text{1A}^0\text{-1a}$\\
$\text{5D}^0\text{-5a}$ & $12$ & $5$ & $\smallmat{4}{0}{1}{4},\smallmat{1}{0}{0}{2} $ & {\small $125t \,/\, (t^2-11t-1)$} & $\text{5B}^0\text{-5a}$\\
$\text{5D}^0\text{-5b}$ & $12$ & $5$ & $\smallmat{4}{0}{1}{4},\smallmat{2}{0}{0}{1} $ & {\small $(t^2-11t-1) \,/\, t$} & $\text{5B}^0\text{-5a}$\\
$\text{5E}^0\text{-5a}$ & $15$ & $5$ & $\smallmat{2}{1}{0}{3},\smallmat{2}{0}{2}{3},\smallmat{1}{0}{2}{2} $ & {\small $(t+5)(t^2-5) \,/\, (t^2+5t+5)$} & $\text{5A}^0\text{-5a}$\\
\bottomrule
\end{tabular}
\bigskip

\caption{Genus zero groups of $\ell$-power level for primes $\ell>3$}\label{table:g0odd}
\end{table}
\clearpage

\begin{center}
\setlength{\extrarowheight}{0.25pt}
\fontsize{9.25pt}{9.25pt}
\begin{tabular}{l@{\hskip 7pt}r@{\hskip 7pt}r@{\hskip 10pt}l@{\hskip 7pt}l@{\hskip 7pt}l}
label & $i$ & $N$ & generators & map & supergroups\\\toprule
$\text{1A}^0\text{-1a}$ & $1$ & $1$ & & &\\
$\text{2A}^0\text{-2a}$ & $2$ & $2$ & $\smallmat{0}{1}{1}{1} $ & {\scriptsize $t^2+1728$} & $\text{1A}^0\text{-1a}$\\
$\text{2A}^0\text{-4a}$ & $2$ & $4$ & $\smallmat{1}{2}{0}{1},\smallmat{1}{1}{1}{2},\smallmat{1}{1}{0}{3} $ & {\scriptsize $-t^2+1728$} &$\text{1A}^0\text{-1a}$ \\
$\text{2A}^0\text{-8a}$ & $2$ & $8$ & $\smallmat{1}{2}{0}{1},\smallmat{1}{1}{1}{2},\smallmat{1}{0}{0}{3},\smallmat{1}{1}{0}{5} $ & {\scriptsize $-2t^2+1728$} & $\text{1A}^0\text{-1a}$\\
$\text{2A}^0\text{-8b}$ & $2$ & $8$ & $\smallmat{1}{2}{0}{1},\smallmat{1}{1}{1}{2},\smallmat{1}{1}{0}{3},\smallmat{1}{1}{0}{5} $ & {\scriptsize $2t^2+1728$} & $\text{1A}^0\text{-1a}$\\
$\text{2B}^0\text{-2a}$ & $3$ & $2$ & $\smallmat{0}{1}{1}{0} $ & {\scriptsize $(256-t)^3 / t^2$} & $\text{1A}^0\text{-1a}$\\
$\text{2C}^0\text{-2a}$ & $6$ & $2$ & $ $ & {\scriptsize $-t^2+64$} & $\text{2B}^0\text{-2a}$,\,$\text{2A}^0\text{-2a}$\\
\bottomrule
\end{tabular}
\end{center}

\begin{center}
\setlength{\extrarowheight}{0.8pt}
\fontsize{9.25pt}{9.25pt}
\begin{tabular}{l@{\hskip 8pt}r@{\hskip 8pt}r@{\hskip 10pt}l@{\hskip 8pt}l@{\hskip 8pt}l}
label & $i$ & $N$ & generators & map & supergroups\\\toprule
$\text{8B}^0\text{-8a}$ & $12$ & $8$ & $\smallmat{3}{0}{0}{3},\smallmat{0}{3}{5}{0},\smallmat{1}{2}{2}{5},\smallmat{1}{0}{0}{3},\smallmat{1}{0}{0}{5} $ & {\scriptsize $16t^2$} & $\text{4C}^0\text{-4a}$\\
$\text{8B}^0\text{-8b}$ & $12$ & $8$ & $\smallmat{3}{0}{0}{3},\smallmat{0}{3}{5}{0},\smallmat{1}{2}{2}{5},\smallmat{1}{0}{0}{3},\smallmat{1}{4}{0}{5} $ & {\scriptsize $32t^2$} & $\text{4C}^0\text{-4a}$\\
$\text{8B}^0\text{-8c}$ & $12$ & $8$ & $\smallmat{3}{0}{0}{3},\smallmat{0}{3}{5}{0},\smallmat{1}{2}{2}{5},\smallmat{3}{2}{0}{1},\smallmat{1}{0}{0}{5} $ & {\scriptsize $32t^2$} & $\text{4C}^0\text{-4b}$\\
$\text{8B}^0\text{-8d}$ & $12$ & $8$ & $\smallmat{3}{0}{0}{3},\smallmat{0}{3}{5}{0},\smallmat{1}{2}{2}{5},\smallmat{3}{2}{0}{1},\smallmat{1}{4}{0}{5} $ & {\scriptsize $16t^2$} & $\text{4C}^0\text{-4b}$\\
$\text{8C}^0\text{-8a}$ & $12$ & $8$ & $\smallmat{3}{0}{0}{3},\smallmat{5}{0}{0}{5},\smallmat{0}{3}{5}{2},\smallmat{1}{2}{0}{3},\smallmat{1}{0}{0}{5} $ & {\scriptsize $-8t^2$} & $\text{4B}^0\text{-4b}$\\
$\text{8C}^0\text{-8b}$ & $12$ & $8$ & $\smallmat{3}{0}{0}{3},\smallmat{5}{0}{0}{5},\smallmat{0}{3}{5}{2},\smallmat{1}{2}{0}{3},\smallmat{1}{4}{0}{5} $ & {\scriptsize $-4(t^2+4)$} & $\text{4B}^0\text{-4b}$\\
$\text{8C}^0\text{-8c}$ & $12$ & $8$ & $\smallmat{3}{0}{0}{3},\smallmat{5}{0}{0}{5},\smallmat{0}{3}{5}{2},\smallmat{3}{2}{0}{1},\smallmat{1}{0}{0}{5} $ & {\scriptsize $-8(t^2+2)$} & $\text{4B}^0\text{-4b}$\\
\bottomrule
\end{tabular}
\end{center}

\begin{center}
\setlength{\extrarowheight}{1.4pt}
\fontsize{9.25pt}{9.25pt}
\begin{tabular}{l@{\hskip 6pt}r@{\hskip 6pt}r@{\hskip 10pt}l@{\hskip 6pt}l@{\hskip 6pt}l}
label & $i$ & $N$ & generators & map & supergroups\\\toprule
$\text{8J}^0\text{-8a}$ & $24$ & $8$ & $\smallmat{3}{2}{0}{3},\smallmat{5}{2}{0}{5},\smallmat{1}{2}{4}{1},\smallmat{1}{0}{0}{3},\smallmat{1}{0}{0}{5} $ & {\scriptsize $(t^2+2) / t^2$} & $\text{4E}^0\text{-4c}$\\
$\text{8J}^0\text{-8b}$ & $24$ & $8$ & $\smallmat{3}{2}{0}{3},\smallmat{5}{2}{0}{5},\smallmat{1}{2}{4}{1},\smallmat{1}{0}{0}{3},\smallmat{1}{2}{0}{5} $ & {\scriptsize $t^2-1$} & $\text{4E}^0\text{-4c}$\\
$\text{8J}^0\text{-8c}$ & $24$ & $8$ & $\smallmat{3}{2}{0}{3},\smallmat{5}{2}{0}{5},\smallmat{1}{2}{4}{1},\smallmat{1}{2}{0}{3},\smallmat{1}{0}{0}{5} $ & {\scriptsize $t^2 / (t^2-2)$} & $\text{4E}^0\text{-4c}$\\
$\text{8J}^0\text{-8d}$ & $24$ & $8$ & $\smallmat{3}{2}{0}{3},\smallmat{5}{2}{0}{5},\smallmat{1}{2}{4}{1},\smallmat{1}{2}{0}{3},\smallmat{1}{2}{0}{5} $ & {\scriptsize $t^2+1$} & $\text{4E}^0\text{-4c}$\\
$\text{8K}^0\text{-16a}$ & $24$ & $16$ & $\smallmat{1}{4}{0}{1},\smallmat{7}{0}{0}{7},\smallmat{0}{3}{5}{0},\smallmat{3}{0}{0}{5},\smallmat{1}{2}{2}{1} $ & {\scriptsize $(t^2+4) / 2$} & $\text{4F}^0\text{-8b}$\\
$\text{8K}^0\text{-16b}$ & $24$ & $16$ & $\smallmat{1}{4}{0}{1},\smallmat{7}{0}{0}{7},\smallmat{0}{3}{5}{0},\smallmat{7}{0}{0}{9},\smallmat{1}{2}{2}{1} $ & {\scriptsize $(t^2-4) / 2$} & $\text{4F}^0\text{-8b}$\\
$\text{8K}^0\text{-16c}$ & $24$ & $16$ & $\smallmat{1}{4}{0}{1},\smallmat{7}{0}{0}{7},\smallmat{2}{3}{1}{2},\smallmat{3}{0}{0}{5},\smallmat{1}{2}{2}{1} $ & {\scriptsize $t^2+2$} & $\text{4F}^0\text{-8b}$\\
\bottomrule
\end{tabular}
\end{center}

\begin{table}
\begin{center}
\setlength{\extrarowheight}{0.4pt}
\fontsize{9.75pt}{9.75pt}
\begin{tabular}{l@{\hskip 6pt}r@{\hskip 6pt}r@{\hskip 8pt}l@{\hskip 6pt}l@{\hskip 6pt}l}
label & $i$ & $N$ & generators & map & supergroups\\\toprule
$\text{16C}^0\text{-16a}$ & $24$ & $16$ & $\smallmat{7}{0}{0}{7},\smallmat{3}{8}{0}{11},\smallmat{0}{3}{5}{2},\smallmat{1}{2}{0}{3},\smallmat{1}{4}{0}{5} $ & {\scriptsize $t^2 / 2$} & $\text{8C}^0\text{-8b}$\\
$\text{16C}^0\text{-16b}$ & $24$ & $16$ & $\smallmat{7}{0}{0}{7},\smallmat{3}{8}{0}{11},\smallmat{0}{3}{5}{2},\smallmat{1}{2}{0}{3},\smallmat{5}{4}{0}{1} $ & {\scriptsize $t^2$} & $\text{8C}^0\text{-8b}$\\
$\text{16C}^0\text{-16c}$ & $24$ & $16$ & $\smallmat{7}{0}{0}{7},\smallmat{3}{8}{0}{11},\smallmat{0}{3}{5}{2},\smallmat{1}{4}{0}{5},\smallmat{5}{2}{0}{3} $ & {\scriptsize $2t^2$} & $\text{8C}^0\text{-8d}$\\
$\text{16C}^0\text{-16d}$ & $24$ & $16$ & $\smallmat{7}{0}{0}{7},\smallmat{3}{8}{0}{11},\smallmat{0}{3}{5}{2},\smallmat{5}{4}{0}{1},\smallmat{1}{6}{0}{3} $ & {\scriptsize $t^2$} & $\text{8C}^0\text{-8d}$\\
$\text{16D}^0\text{-16a}$ & $24$ & $16$ & $\smallmat{7}{0}{0}{7},\smallmat{3}{0}{0}{11},\smallmat{0}{3}{5}{2},\smallmat{1}{4}{0}{5},\smallmat{5}{2}{0}{3} $ & {\scriptsize $t^2-4$} & $\text{8C}^0\text{-8d}$\\
$\text{16D}^0\text{-16b}$ & $24$ & $16$ & $\smallmat{7}{0}{0}{7},\smallmat{3}{0}{0}{11},\smallmat{0}{3}{5}{2},\smallmat{3}{2}{0}{1},\smallmat{1}{4}{0}{5} $ & {\scriptsize $t^2+4$} & $\text{8C}^0\text{-8d}$\\
$\text{16D}^0\text{-16c}$ & $24$ & $16$ & $\smallmat{7}{0}{0}{7},\smallmat{3}{0}{0}{11},\smallmat{0}{3}{5}{2},\smallmat{3}{2}{0}{1},\smallmat{5}{2}{0}{3} $ & {\scriptsize $2(t^2-2)$} & $\text{8C}^0\text{-8d}$\\
\bottomrule
\end{tabular}
\end{center}
\bigskip
\caption{Genus 0 groups of 2-power level.}\label{table:g0two}
\end{table}

\pagebreak

\begin{table}
\begin{center}
\setlength{\extrarowheight}{3pt}
\fontsize{9.25pt}{9.25pt}
\begin{tabular}{l@{\hskip 8pt}r@{\hskip 8pt}r@{\hskip 8pt}l@{\hskip 8pt}l@{\hskip 8pt}l@{\hskip 8pt}l}
group & $i$ & $N$ & generators & curve & map & supergroups\\\toprule
$\text{16C}^1\text{-16c}$ & $24$ & $16$ & $\smallmat{2}{1}{3}{2},\smallmat{0}{3}{5}{8},\smallmat{1}{0}{0}{5},\smallmat{1}{8}{0}{3} $ & \href{http://www.lmfdb.org/EllipticCurve/Q/256a2}{\texttt{256a2}} & {\scriptsize $(x-1) / 2$} & $\text{8D}^0\text{-8a}$\\
$\text{16C}^1\text{-16d}$ & $24$ & $16$ & $\smallmat{2}{1}{3}{2},\smallmat{0}{3}{5}{8},\smallmat{3}{0}{0}{5},\smallmat{1}{8}{0}{3} $ & \href{http://www.lmfdb.org/EllipticCurve/Q/256a1}{\texttt{256a1}} & {\scriptsize $x+1$} & $\text{8D}^0\text{-8a}$\\
$\text{16B}^1\text{-16a}$ & $24$ & $16$ & $\smallmat{3}{0}{0}{11},\smallmat{0}{3}{5}{0},\smallmat{2}{3}{9}{6},\smallmat{1}{0}{0}{3},\smallmat{1}{0}{0}{5} $ & \href{http://www.lmfdb.org/EllipticCurve/Q/256b2}{\texttt{256b2}} & {\scriptsize $x / 4$} & $\text{8B}^0\text{-8a}$\\
$\text{16B}^1\text{-16c}$ & $24$ & $16$ & $\smallmat{3}{0}{0}{11},\smallmat{0}{3}{5}{0},\smallmat{2}{3}{9}{6},\smallmat{1}{4}{0}{3},\smallmat{1}{0}{0}{5} $ & \href{http://www.lmfdb.org/EllipticCurve/Q/256b1}{\texttt{256b1}} & {\scriptsize $x / 2$} & $\text{8B}^0\text{-8a}$\\
$\text{16I}^1\text{-16d}$ & $48$ & $16$ & $\smallmat{3}{0}{0}{11},\smallmat{0}{3}{5}{8},\smallmat{1}{4}{12}{1},\smallmat{1}{0}{0}{5},\smallmat{1}{8}{0}{3} $ & \href{http://www.lmfdb.org/EllipticCurve/Q/256a1}{\texttt{256a1}} & {\scriptsize $x-1$} & $\text{8H}^0\text{-8a}$,\,$\text{16E}^0\text{-16c}$,\,$\text{16C}^1\text{-16c}$\\
$\text{16I}^1\text{-16f}$ & $48$ & $16$ & $\smallmat{3}{0}{0}{11},\smallmat{0}{3}{5}{8},\smallmat{1}{4}{12}{1},\smallmat{1}{8}{0}{3},\smallmat{2}{3}{5}{2} $ & \href{http://www.lmfdb.org/EllipticCurve/Q/256a2}{\texttt{256a2}} & $\frac{4x-4y+12}{x^2-2x-15}$ & $\text{8H}^0\text{-8b}$,\,$\text{16E}^0\text{-16c}$,\,$\text{16C}^1\text{-16d}$\\
\bottomrule
\end{tabular}
\end{center}
\bigskip
\caption{Genus $1$ groups $G$ of prime power level for which $X_G(\QQ)$ is infinite}  \label{table:g1}
\end{table}

\begin{table}
\begin{center}
\begin{tabular}{ll}
label &  Weierstrass model\\\toprule
\href{http://www.lmfdb.org/EllipticCurve/Q/121b1}{\texttt{121b1}} & $y^2+y=x^3-x^2-7x+10$\\
\href{http://www.lmfdb.org/EllipticCurve/Q/128a1}{\texttt{128a1}} & $y^2=x^3+x^2+x+1$\\
\href{http://www.lmfdb.org/EllipticCurve/Q/128a2}{\texttt{128a2}} & $y^2=x^3+x^2-9x+7$\\
\href{http://www.lmfdb.org/EllipticCurve/Q/256a1}{\texttt{256a1}} & $y^2=x^3+x^2-3x+1$\\
\href{http://www.lmfdb.org/EllipticCurve/Q/256a1}{\texttt{256a2}} & $y^2=x^3+x^2-13x-21$\\
\href{http://www.lmfdb.org/EllipticCurve/Q/256a1}{\texttt{256b1}} & $y^2=x^3-2x$\\
\href{http://www.lmfdb.org/EllipticCurve/Q/256a1}{\texttt{256b2}} & $y^2=x^3+8x$\\
\bottomrule
\end{tabular}
\end{center}
\bigskip
\caption{Some elliptic curves}  \label{table:ec}
\end{table}

\noindent
\[
J(x,y):=\frac{(f_1 f_2 f_3 f_4)^3}{f_5^2 f_6^{11}},
\]
\begin{align*}
f_1&=x^2+3x-6,
&f_2&=11(x^2-5)y+(2x^4+23x^3-72x^2-28x+127),\\
f_3&=6y+11x-19,
&f_4&=22(x-2)y+(5x^3+17x^2-112x+120), \\
f_5&=11y+(2x^2+17x-34), 
&f_6&=(x-4)y-(5x-9).
\end{align*}

\newpage
\bibliographystyle{plain}

\begin{bibdiv}
\begin{biblist}

\bib{Magma}{article}{
      author={Bosma, Wieb},
      author={Cannon, John},
      author={Playoust, Catherine},
       title={The {M}agma algebra system. {I}. {T}he user language},
        date={1997},
     journal={J. Symbolic Comput.},
      volume={24},
      number={3-4},
       pages={235\ndash 265},
        note={Computational algebra and number theory (London, 1993)},
}

\bib{MR2059637}{article}{
   author={Chua, Kok Seng},
   author={Lang, Mong Lung},
   author={Yang, Yifan},
   title={On Rademacher's conjecture: congruence subgroups of genus zero of
   the modular group},
   journal={J. Algebra},
   volume={277},
   date={2004},
   number={1},
   pages={408--428},
   issn={0021-8693},
}

\bib{Cremona}{misc}{
	author={Cremona, John E.},
	title={Elliptic curve data},
	date={2014},
	note={\url{http://homepages.warwick.ac.uk/staff/J.E.Cremona/ftp/data/}},
}

\bib{MR2016709}{article}{
      author={Cummins, C.~J.},
      author={Pauli, S.},
       title={Congruence subgroups of {${\rm PSL}(2,{\mathbb Z})$} of genus less
  than or equal to $24$},
        date={2003},
        ISSN={1058-6458},
     journal={Experiment. Math.},
      volume={12},
      number={2},
       pages={243\ndash 255},
         url={http://projecteuclid.org/getRecord?id=euclid.em/1067634734},
      review={\MR{MR2016709 (2004i:11037)}},
}

\bib{CPwebsite}{misc}{
      author={Cummins, C.~J.},
      author={Pauli, S.},
       title={Database of congruence subgroups of {${\rm PSL}(2,{\mathbb Z})$} of genus less
  than or equal to $24$},
	    note={available at \url{http://www.uncg.edu/mat/faculty/pauli/congruence/}}
}

\bib{MR0337993}{incollection}{
      author={Deligne, P.},
      author={Rapoport, M.},
       title={Les sch\'emas de modules de courbes elliptiques},
        date={1973},
   booktitle={Modular functions of one variable, {II} ({P}roc. {I}nternat.
  {S}ummer {S}chool, {U}niv. {A}ntwerp, {A}ntwerp, 1972)},
   publisher={Springer},
     address={Berlin},
       pages={143\ndash 316. Lecture Notes in Math., Vol. 349},
      review={\MR{MR0337993 (49 \#2762)}},
}
\bib{MR0342466}{article}{
    author = {Dennin, Jr., Joseph B.},
     title = {The genus of subfields of {$K(p^{n})$}},
   journal = {Illinois J. Math.},
    volume = {18},
      date = {1974},
     pages = {246--264},
      issn = {0019-2082},
    review = {\MR{MR0342466 (49 \#7212)}},
}


\bib{MR718935}{article}{
   author={Faltings, G.},
   title={Endlichkeitss\"atze f\"ur abelsche Variet\"aten \"uber
   Zahlk\"orpern},
   language={German},
   journal={Invent. Math.},
   volume={73},
   date={1983},
   number={3},
   pages={349--366},
   issn={0020-9910},
   review={\MR{718935 (85g:11026a)}},
   doi={10.1007/BF01388432},
}

\bib{MR1677158}{article}{
   author={Halberstadt, Emmanuel},
   title={Sur la courbe modulaire $X_{\text{nd\'ep}}(11)$},
   language={French, with English and French summaries},
   journal={Experiment. Math.},
   volume={7},
   date={1998},
   number={2},
   pages={163--174},
   issn={1058-6458},
   review={\MR{1677158 (99m:11062)}},
}

\bib{MR648603}{book}{
   author={Kubert, Daniel S.},
   author={Lang, Serge},
   title={Modular units},
   series={Grundlehren der Mathematischen Wissenschaften [Fundamental
   Principles of Mathematical Science]},
   volume={244},
   publisher={Springer-Verlag, New York-Berlin},
   date={1981},
   pages={xiii+358},
   isbn={0-387-90517-0},
   review={\MR{648603 (84h:12009)}},
}

\bib{LMFDB}{misc}{
  author       = {{LMFDB Collaboration}},
  title        = {The L-functions and Modular Forms Database},
  note         = {\url{http://www.lmfdb.org}},
  year         = {2015},
}


\bib{RZB}{article}{
	author={Rouse, Jeremy},
	author={Zureick-Brown, David},
	title={Elliptic curves over $\QQ$ and $2$-adic images of Galois},
	journal={Research in Number Theory},
	volume={1},
	date={2015},
}

\bib{MR0387283}{article}{
      author={Serre, Jean-Pierre},
       title={Propri\'et\'es galoisiennes des points d'ordre fini des courbes
  elliptiques},
        date={1972},
        ISSN={0020-9910},
     journal={Invent. Math.},
      volume={15},
      number={4},
       pages={259\ndash 331},
      review={\MR{MR0387283 (52 \#8126)}},
}

\bib{MR644559}{article}{
      author={Serre, Jean-Pierre},
       title={Quelques applications du th\'eor\`eme de densit\'e de
  {C}hebotarev},
        date={1981},
        ISSN={0073-8301},
     journal={Inst. Hautes \'Etudes Sci. Publ. Math.},
      number={54},
       pages={323\ndash 401},
      review={\MR{MR644559 (83k:12011)}},
}

\bib{MR1043865}{book}{
      author={Serre, Jean-Pierre},
       title={Abelian {$l$}-adic representations and elliptic curves},
     edition={Second Ed.},
      series={Advanced Book Classics},
   publisher={Addison-Wesley Publishing Company Advanced Book Program},
     address={Redwood City, CA},
        date={1989},
        ISBN={0-201-09384-7},
        note={With the collaboration of Willem Kuyk and John Labute},
      review={\MR{MR1043865 (91b:11071)}},
}

\bib{MR1757192}{book}{
      author={Serre, Jean-Pierre},
       title={Lectures on the {M}ordell-{W}eil theorem},
     edition={Third Ed.},
      series={Aspects of Mathematics},
   publisher={Friedr. Vieweg \& Sohn},
     address={Braunschweig},
        date={1997},
        note={Translated from the French and edited by Martin Brown from notes
  by Michel Waldschmidt, With a foreword by Brown and Serre},
      review={\MR{MR1757192 (2000m:11049)}},
}


\bib{MR1291394}{book}{
   author={Shimura, Goro},
   title={Introduction to the arithmetic theory of automorphic functions},
   series={Publications of the Mathematical Society of Japan},
   volume={11},
   note={Reprint of the 1971 original;
   Kan\^o Memorial Lectures, 1},
   publisher={Princeton University Press, Princeton, NJ},
   date={1994},
   pages={xiv+271},
   isbn={0-691-08092-5},
   review={\MR{1291394 (95e:11048)}},
}

\bib{ComputingImages}{article}{
author={Sutherland, Andrew V.},
title={Computing images of Galois representations attached to elliptic curves},
journal={Forum of Mathematics, Sigma},
volume={4},
date={2016},
pages={79 pages},
}

\bib{MagmaScripts}{misc}{
author={Sutherland, Andrew V.},
author={Zywina, David},
title={{\rm \texttt{Magma}} scripts associated to ``Modular curves of prime-power level with infinitely many rational points"},
date={2016},
note={available at \url{http://math.mit.edu/~drew/SZ16}},
}

\bib{PossibleImages}{misc}{
author={Zywina, David},
title={On the possible images of the mod $\ell$ representations associated to elliptic curves over $\QQ$},
date={2015},
note={arXiv:1508.07660v1 [math.NT]},
}

\bib{PossibleIndices}{misc}{
author={Zywina, David},
title={Possible indices for the Galois image of elliptic curves over $\QQ$},
date={2015},
note={arXiv:1508.07663v1 [math.NT]},
}

\end{biblist}
\end{bibdiv}

\end{document}